\documentclass[reqno,11pt]{amsart}

\usepackage{tikz}
\usepackage{tikz-cd}
\usepackage{amssymb}
\usepackage{amsgen}
\usepackage{amsmath}
\usepackage{amsthm}
\usepackage{mathrsfs}
\usepackage{cite}
\usepackage{amsfonts}

\newcommand{\sour}{\mathop{\boldsymbol s}}

\newcommand{\skel}[1]{^{(#1)}}
\newcommand{\Iso}{\mathop{\mathrm{Iso}}}

\newcommand{\ran}{\mathop{\boldsymbol r}\nolimits}

\newcommand{\inv}{^{-1}}
\newcommand{\p}{\varphi}

\newcommand{\wh}{\widehat}

\usepackage{xcolor}

\newtheorem{Thm}{Theorem}[section]
\newtheorem{Prop}[Thm]{Proposition}

\newtheorem{Lemma}[Thm]{Lemma}
{\theoremstyle{definition}
\newtheorem{Def}[Thm]{Definition}}
{\theoremstyle{remark}
\newtheorem{Rmk}[Thm]{Remark}}
\newtheorem{Cor}[Thm]{Corollary}
{\theoremstyle{remark}
}
{\theoremstyle{remark}
\newtheorem{Example}[Thm]{Example}}

{\theoremstyle{remark}
}
{\theoremstyle{remark}
}
{\theoremstyle{remark}
}
{\theoremstyle{remark}
\newtheorem*{Claim*}{Claim}}
\newtheorem*{thmA}{Theorem~A}
\newtheorem*{thmB}{Theorem~B}

\numberwithin{equation}{section}

\title{On von Neumann regularity of ample groupoid algebras}

\author{Benjamin Steinberg\and Daniel W.~van Wyk}
\address[B.~Steinberg]{%
    Department of Mathematics\\
    City College of New York\\
    Convent Avenue at 138th Street\\
    New York, New York 10031\\
    USA}
\email{bsteinberg@ccny.cuny.edu}
\address[D. W.~van Wyk]{%
    Department of Mathematics \\ Fairfield University\\ Fairfield, CT
06824 \\ USA, \and Department of Mathematics and Applied Mathematics\\ University of
the Free State \\ Park West, Bloemfontein, 9301\\  South Africa.}
\email{dvanwyk@fairfield.edu}

\thanks{The first author was supported by a Simons Foundation Collaboration Grant, award number 849561, and the Australian Research Council Grant DP230103184.}
\date{\today}


\begin{document}

\begin{abstract}
We completely characterize when the algebra of an ample groupoid with coefficients in an arbitrary unital ring is von Neumann regular and, more generally, when the algebra of a graded ample groupoid is graded von Neumann regular.  Our main application is to resolve the question, open since 1970, of when the algebra of an inverse semigroup is von Neumann regular.   As applications we recover known results on regularity and graded regularity of Leavitt path algebras, and prove a number of new results, in particular concerning graded regularity of algebras of Deaconu-Renault groupoids and Nekrashevych-Exel-Pardo algebras of self-similar groups.
\end{abstract}
\maketitle

\section{Introduction}

The notion of a (von Neumann) regular ring was invented by von Neumann in the 1930s in his study of what are now called von Neumann algebras. A ring (or more generally a semigroup) is regular if for every $a$, there is a $b$ with $aba=a$.   It was later observed that regular rings are precisely the rings for which every module is flat, and so from the point of view of homological algebra they are a natural generalization of semisimple rings.

A famous result of Connell~\cite{connell} says that a group ring $RG$, over a unital ring $R$, is regular if and only if $R$ is regular, $G$ is locally finite and the order of every finite subgroup of $G$ is invertible in $R$.  In particular, this says that regularity is a rare feature for groups rings.  Leavitt path algebras~\cite{LeavittBook} are also rarely regular: they are regular if and only if the (directed) graph defining the Leavitt path algebra is acyclic~\cite{AbramRangaswamy10}.  However, every Leavitt path algebra over a field is graded von Neumann regular by a result of Hazrat~\cite{Hazrat14}.  A graded ring $R$ is graded regular if the semigroup of homogeneous elements is regular.  This has an important impact on the graded $K$-theory of Leavitt path algebras~\cite{HazratGraded}.

Okni\'nski tried to generalize Connell's results to inverse semigroup algebras (see his book~\cite{okninski}).  He was able to prove that if $K$ is a field of characteristic $0$ and $S$ is an inverse semigroup, then $KS$ is regular if and only if $S$ is locally finite.  In positive characteristic he was able to give necessary conditions: $S$ is periodic, each maximal subgroup of $S$ is locally finite and the characteristic of $K$ cannot divide the order of any element of a maximal subgroup.  It had been proven earlier by Weissglass~\cite{Weissglass} that if $R$ is regular, $S$ is locally finite and the order of every finite subgroup of $S$ is invertible in $R$, then $RS$ is regular.  But being periodic with locally finite maximal subgroups is a strictly weaker condition than being locally finite.  The reason that Okni\'nski was able to obtain a complete characterization in characteristic $0$ is that he could embed $\mathbb CS$ into the Banach algebra $\ell_1(S)$ and use some analysis.

Group algebras, inverse semigroup algebras and Leavitt path algebras are all special cases of algebras of ample groupoids~\cite{mygroupoidalgebra}.  It seems natural then to investigate regularity of ample groupoid algebras.  A first result in this direction was obtained in~\cite{AHL}, where it was show that if $K$ is a field and $\mathscr G$ is an ample groupoid, then a necessary condition for the groupoid algebra $K\mathscr G$ to be regular is that each isotropy group of $\mathscr G$ is locally finite and the order of no element of any isotropy group is divisible by the characteristic of $K$.   The proof is a reduction to Connell's result.

Our main result is a complete characterization of regularity of $R\mathscr G$ over any base unital ring $R$.  Following the Bourbaki tradition, we say a space is quasi-compact if every open covering has a finite subcover and reserve ``compact'' to mean quasi-compact and Hausdorff.

\begin{thmA}
Let $\mathscr G$ be an ample groupoid and $R$ a unital ring.  Then $R\mathscr G$ is von Neumann regular if and only if:
\begin{enumerate}
  \item $R$ is von Neumann regular;
  \item $\mathscr G$ is a directed union of quasi-compact open subgroupoids;
  \item the order of each finite subgroup of an isotropy group of $\mathscr G$ is invertible in $R$.
\end{enumerate}
\end{thmA}

 For principal, second countable Hausdorff groupoids the second condition is equivalent to approximate finiteness in the sense of~\cite{GPS}, hence we call such groupoids approximately quasi-compact.

As a consequence, we resolve the problem of characterizing regular inverse semigroup algebras, dating back to the 1970 paper of Weissglass~\cite{Weissglass}.

\begin{thmB}
Let $S$ be an inverse semigroup and $R$ a unital ring.  Then $RS$ is von Neumann regular if and only if:
\begin{enumerate}
  \item $R$ is von Neumann regular;
  \item $S$ is locally finite;
  \item the order of each finite subgroup of $S$ is invertible in $R$.
\end{enumerate}
\end{thmB}

The sufficiency of these conditions goes back to Weissglass~\cite{Weissglass}.

We also consider graded regularity of groupoid algebras for ample groupoids equipped with a locally constant cocycle to a group.  We prove the algebra is graded regular if and only if the homogeneous component of the identity is regular.  Since the homogeneous component of the identity is the algebra of the clopen subgroupoid of elements mapping to the identity under the cocycle, our results for regularity of groupoid algebras apply.  From this we easily recover Hazrat's result on graded regularity of Leavitt path algebras~\cite{Hazrat14}.  We extend these results to algebras of Deaconu-Renault groupoids under mild hypotheses.  These include algebras of row finite higher rank graphs.  We consider further graded regularity of partial group skew products, Leavitt path algebras of labeled graphs and Nekrashevych-Exel-Pardo algebras of self-similar graphs.

The paper is organized as follows.  Section~2 is devoted to preliminaries concerning inverse semigroups, ample groupoids and regular rings.  Then Section~3 proves our main results on regularity of ample groupoid algebras.  Section~4 addresses the question of graded regularity of the algebra of an ample groupoid equipped with a locally constant cocycle.  The final section, applies the main results to study regularity and graded regularity for the algebras of a number of well-studied families of ample groupoids.  We added an appendix giving a groupoid-free proof of Theorem~B for specialists in semigroup theory, although the proof is merely a retranslation of the groupoid proof.

\subsection*{Acknowledgments}
We thank Gilles de Castro for showing Example~\ref{e:diff_algs} to us.

\section{Preliminaries}

We collect here some preliminaries on inverse semigroups, ample groupoids and von Neumann regular rings.

\subsection{Inverse semigroups and groupoids}

The reader is referred to~\cite{Lawson} for details on the theory of inverse semigroups.
A \emph{semigroup} is a nonempty set together with an associative binary operation. An \emph{inverse semigroup} is a semigroup $S$ such that for every $s\in S$ there is a unique $s^*\in S$ such that $s^*ss^*=s^*$ and $ss^*s=s$. Notice that $ss^*, s^*s$ are idempotents for all $s\in S$. The set $E(S)$ of idempotents  of $S$ is a commutative subsemigroup and hence a meet semilattice with the partial order given by $e\leq f$ if and only if $ef=e$, and the meet by $e\wedge f =ef$. The order extends to $S$ by defining $s\leq t$ if $s=te$ for some $e\in E(S)$. If $e\in E(S)$, then $G_e=\{s\in S\mid s^*s=e=ss^*\}$ is a group with identity $e$ known as the \emph{maximal subgroup} of $S$ at $e$.  

A topological groupoid is a groupoid $\mathscr G$ endowed with a locally compact topology such that composition and inversion are continuous. All groupoids are assumed nonempty.  We do not assume that $\mathscr G$ is Hausdorff in general, but we do assume that the unit space $\mathscr G^{(0)} = \{\gamma\gamma^{-1}\mid \gamma\in \mathscr G\}$ is Hausdorff with the relative topology inherited from  $\mathscr G$. Since we do not assume that  $\mathscr G$ is Hausdorff, locally compact is defined as every point in $\mathscr G$ having a compact neighborhood, where compact means Hausdorff and open covers have finite subcovers.  Following the Bourbaki tradition, we say that a space $X$ is \emph{quasi-compact} if every open covering of $X$ has a finite subcovering (but with no assumption of being Hausdorff).  Note that a finite union of quasi-compact sets is quasi-compact and a closed subset of a quasi-compact set is quasi-compact.

An \emph{\'etale groupoid} is a topological groupoid  $\mathscr G$ such that the range map $\ran\colon \mathscr G\to \mathscr G^{(0)}$ (or equivalently, the source map $\sour$) is a local homeomorphism.  A \emph{bisection} of an \'etale groupoid $\mathscr G$ is an open subset $U\subseteq \mathscr G$ such that  the restrictions $\ran|_{U}$ and $\sour|_{U}$ are homeomorphisms onto an open set in $\mathscr G\skel 0$. In particular, bisections are always Hausdorff subspaces.   An \'etale groupoid $\mathscr G$ is an \emph{ample groupoid} if  $\mathscr{G}^{(0)}$ is a locally compact Hausdorff space with a basis of compact open sets.

If $x\in \mathscr G\skel 0$, then the \emph{isotropy group} at $x$ is \[\mathscr G^x_x =\{\gamma\in \mathscr G\mid \sour(\gamma)=x=\ran(\gamma)\}.\]  Then $\Iso(\mathscr G)=\bigcup_{x\in \mathscr G\skel 0} \mathscr G^x_x$ is a closed subgroupoid called the \emph{isotropy subgroupoid} (or bundle).  A groupoid consisting of only isotropy is called a \emph{group bundle}.  A groupoid with trivial isotropy groups is called \emph{principal}.

If $X\subseteq \mathscr G\skel 0$, then $\mathscr G|_X$ is the subgroupoid of all arrows $\gamma$ with $\sour(\gamma),\ran(\gamma)\in X$, i.e., $\mathscr G|_X=\sour\inv (X)\cap \ran\inv (X)$.  In particular, if $X$ is open (respectively, closed) in $\mathscr G\skel 0$, then $\mathscr G|_X$ is open (respectively, closed) in $\mathscr G$.

The \emph{orbit} of $x\in \mathscr G\skel 0$ is the set $\mathcal O_x$ consisting of all $y\in \mathscr G\skel 0$ such that there is an arrow $\gamma\colon x\to y$.  Note that $\mathcal O_x= \ran(\sour\inv(x))=\sour(\ran\inv(x))$.  The set of all orbits of $\mathscr G$ is denoted $\mathscr G\skel 0/\mathscr G$.   A subset $X$ of $\mathscr G\skel 0$ is \emph{invariant} if it is a union of orbits.

Suppose that $\mathscr G$ is an ample groupoid. Then the set $\Gamma_c(\mathscr G)$ of compact open bisections form a basis for its topology. Moreover, the set $\Gamma_c(\mathscr G)$ is an inverse semigroup with multiplication and inversion defined by
\begin{align*}
UV =\{\gamma\eta\in \mathscr G\mid \gamma\in U, \eta\in V\}, \text{ and }
U^{-1} = \{\gamma^{-1}\mid \gamma \in U\},
\end{align*}
respectively. The semilattice of idempotents of  $\Gamma_c(\mathscr G)$ is given by the set of compact open subsets of $\mathscr G^{(0)}$.
Notice that if $\mathscr G$ is an ample group bundle, then $\Gamma_c(\mathscr G)$ is a Clifford semigroup, as $\sour(U)=\ran(U)$ for all $U\in \Gamma_c(\mathscr G)$.

\subsection{Algebras of inverse semigroups and ample groupoids}
Fix a unital ring $R$ (not necessarily commutative).  If $S$ is an inverse semigroup, then the inverse semigroup ring $RS$ is the free left $R$-module with basis $S$ and with multiplication \[\sum_{s\in S}c_ss\cdot \sum_{t\in S} d_tt= \sum_{s,t\in S}c_sd_tst.\]

Fix an ample groupoid $\mathscr{G}$.   Then the \emph{Steinberg algebra} of $\mathscr G$ over $R$ is the left $R$-module
\[ R\mathscr{G} = \mathrm{span}_R\{1_U\mid U\in \Gamma_c(\mathscr G)\}\subseteq R^{\mathscr G},\]
 with multiplication given by convolution:
\begin{equation*}
	f\ast g (\gamma) = \sum_{\sour(\gamma)=\sour(\alpha)} f(\gamma\alpha^{-1})g(\alpha)=\sum_{\alpha\beta=\gamma}f(\alpha)g(\beta).
\end{equation*}
Notice that $1_U*1_V = 1_{UV}$ for all $U,V\in \Gamma_c(\mathscr G)$.   Thus $R\mathscr G$ is a quotient of the semigroup algebra $R\Gamma_c(\mathscr G)$.  
Note that $R\mathscr G$ is unital if and only if $\mathscr G\skel 0$ is compact.  See~\cite{mygroupoidalgebra} for details (where the assumption that $R$ is commutative is not necessary).

If $U$ is an open invariant subset of $\mathscr G\skel 0$ and $X=\mathscr G\skel 0\setminus U$ is the complementary closed invariant subspace, then $I=R\mathscr G|_U$ is an ideal of $R\mathscr G$ and $R\mathscr G/I\cong R\mathscr G|_X$ via restriction of functions by~\cite[Pages 1596--1597]{gcrccr} (note an omission in the argument is resolved in~\cite[Corollar~4.4]{MillerSteinberg}), a fact we shall use in the sequel without comment.

\subsection{Groupoid of germs and the universal groupoid of an inverse semigroup}
We briefly review the construction of the universal groupoid of an inverse semigroup; see~\cite{mygroupoidalgebra} or~\cite{Paterson} for a detailed discussion.

Fix an inverse semigroup $S$ and let $E$ denote the semilattice of idempotents of $S$ (since $S$ will be fixed throughout this subsection). Let $X$ be a locally compact totally disconnected space.   We denote by $I_X$ the inverse semigroup of all partial homeomorphisms of $X$ with compact open domain and range. An action of $S$  on $X$ is a homomorphism $\p\colon S\to I_X$ such that if $X_e$ denotes the domain of an idempotent $e$, then $\bigcup_{e\in E}X_e=X$; this last condition says that the action is non-degenerate.  Notice that $\p_s\colon X_{s^*s}\to X_{ss^*}$ for each $s\in S$.  We put $sx=\p_s(x)$ for the action when no confusion can arise.

If $S$ acts on a locally compact totally disconnected space $X$, the \emph{groupoid of germs} $\mathscr G$ of the action is described as follows.   The groupoid $\mathscr G$ consists of all equivalence classes of pairs $(s,x)$ with $x\in X_{s^*s}$ where $(s,x)\sim (t,y)$ if $x=y$ and there exists $u\in S$ with $x\in X_{u^*u}$ and $u\leq s,t$ (equivalently, there exists $e\in E$ with $x\in X_e$ and $se=te$) .   The class of $(s,x)$  is denoted by $[s,x]$.  A basis of compact open bisections for the topology on $\mathscr G$ is given by the sets
\[(s,U)=\{[s,x]\mid x\in U\}\] where $U\subseteq X_{s^*s}$ is compact open.
The set of composable pairs is
$\mathscr{G}_S\skel 2=\{([s,x],[t,y])\mid x=ty, \}$ with multiplication given by
\[[s,ty][t,y] =[st,y], \]
and inversion given by
\[[s,x]^{-1} = [s^*,sx].\]
The unit space $\mathscr{G}_S^{(0)} = \{[e,x]\mid x\in X_e\}$ is identified with $X$ via the homeomorphism $[e,x]\mapsto x$.
From this point of view, the domain and range maps are given by $\sour([s,x])=x$ and $\ran([s,x])=sx$.

There is a homomorphism $\psi\colon S\to \Gamma_c(\mathscr G)$ given by $\psi(s) = (s,X_{s^*s})$ and $\mathscr G=\bigcup \psi(S)$.
The reader should consult~\cite{Exel,Paterson} for more details on groupoids of germs.

Of special importance is the action of $S$ on the spectrum $\widehat{E}$ of $E$.
 Let $\widehat E$ denote the space of nonzero (semi)characters $x\colon E\to \{0,1\}$, endowed with the topology of pointwise convergence. Then $\widehat E$ is a locally compact, totally disconnected Hausdorff space.
If $e, e_1,\ldots e_n\in E$ and $x\in \widehat E$, then
$$D_{e, e_1,\ldots e_n} = \{y\in \widehat E \mid y(e)=1,\,\, y(e_i) = 0,\, 1\leq i\leq n\}$$
is a compact open subset of $\widehat E$, and such sets form a basis $\mathscr{B}$ of compact open sets for $\widehat E$. Basis elements of the form
\[D_e = \{x\in\widehat E \mid x(e)=1 \},\]
(for which $n=0$) play a particularly important role in the construction of the universal groupoid. They allow us to define an action of $S$ on $\widehat E$ as follows.  For each $s\in S$, the map $\alpha_s\colon D_{s^*s}\to D_{ss^*}$ defined by
\[\alpha_s(x)(e) = x(s^*es)\]
is a homeomorphism, and the map $s\mapsto \alpha_s$ defines an action $\alpha\colon S\to I_{\widehat E}$ of $S$ on $\widehat{E}$ (see~\cite[Proposition 4.3.2]{Paterson}\footnote{Note that Paterson works with \emph{right} actions in~\cite{Paterson}.}).

The universal groupoid $\mathscr{G}_S$ associated with $S$ is the groupoid of germs $\mathscr G_S$ of the action of $S$ on $\widehat E$.
For $s\in S$, define $\psi(s) = (s,D_{s^*s})$. Then $\psi$ is a homomorphism of $S$ into $\Gamma_c(\mathscr{G}_S)$,~\cite[Theorem 3.3.2]{Paterson}\footnote{In fact, $\psi$ is a restriction of Paterson's  $\psi_X$ in~\cite[Theorem 3.3.2]{Paterson}}. In fact, $\psi$ is injective, and extends to an isomorphism $\wh{\psi}\colon RS\to R\mathscr{G}_S$; see~\cite[Theorem 6.3]{mygroupoidalgebra} where the assumption that $R$ is commutative is unnecessary.

If $e\in E$ and $x_e$ is the principal character given by \[x_e(f) = \begin{cases}1, & \text{if}\ f\geq e\\ 0, & \text{else.}\end{cases}\] then the isotropy group of $\mathscr G_S$ at $e$ is isomorphic to the maximal subgroup $G_e$.  The principal characters form a dense subset of $\wh{E}$.  See~\cite{Paterson,mygroupoidalgebra} for more details.

\subsection{Regular rings}
A ring $A$ is \emph{von Neumann regular}, or regular for short, if for all $a\in A$, there exists $b\in A$ with $aba=a$.  We do not require rings to be unital, however the coefficient rings of our group, semigroup and groupoid algebras shall always be assumed unital.   The following facts are quite standard; see~\cite[Chapter~1]{Goodearlreg}.

\begin{Prop}\label{p:vnr}
The following statements all hold.
\begin{enumerate}
  \item An arbitrary direct product of regular rings is regular.
  \item If $A$ is a ring and $I$ is an ideal of $A$, then $A$ is regular if and only if $I$ and $A/I$ are regular.
  \item A direct limit of regular rings is regular.
  \item A corner in a regular ring is regular.
  \item A matrix algebra over a regular ring is regular.
  \item The center of a regular ring is regular.
  \item Every semisimple Artinian ring is regular.
  \item A ring $A$ is regular if and only if it has local units and each unitary $A$-module is flat.
\end{enumerate}
\end{Prop}

A semigroup is \emph{locally finite} if all its finitely generated subsemigroups are finite, or, equivalently, it is a direct limit of finite semigroups.  Note that any subsemigroup of a locally finite semigroup is locally finite.
The following is a special case of a classical theorem of Birkhoff in universal algebra, cf.~\cite[Theorem~10.16]{universalalgebra}.  We include a proof for the sake of completeness.

\begin{Thm}[Birkhoff]\label{t:Birkhoff}
Let $S$ be a subsemigroup of a direct product $\prod_{\alpha \in A} S_{\alpha}$ of finite semigroups of uniformly bounded order.  Then $S$ is locally finite.
\end{Thm}
\begin{proof}
Without loss of generality, we may assume that $S$ is generated by a finite set $X$.  Suppose that $|S_{\alpha}|\leq d$ for all $\alpha\in A$.  Let $T_{d+1}$ be the (finite) monoid of all self-maps on a $(d+1)$-element set.  Let $S_{\alpha}^1$ be the monoid obtained by adding an identity to $S_{\alpha}$.  Then the action of $S_{\alpha}$ on $S_{\alpha}^1$ by left multiplication is faithful and so $S_{\alpha}$ embeds in $T_{d+1}$ as $|S_{\alpha}^1|\leq d+1$.  Hence combining the projections $\pi_{\alpha}\colon S\to S_{\alpha}$ with embeddings into $T_{d+1}$, we see that the homomorphisms from $S$ to $T_{d+1}$ separate points of $S$.  But since $X$ is finite and a homomorphism from $S$ to $T_{d+1}$ is determined by its restriction to $X$, we see that $S$ has at most $r=|T_{d+1}|^{|X|}$ many distinct homomorphisms to $T_{d+1}$.  Thus $S$ embeds in $T_{d+1}^r$ (via the product of these homomorphisms) and hence is finite.
\end{proof}

The following result is due to Connell~\cite{connell}.

\begin{Thm}[Connell]\label{t:connell}
Let $G$ be a group and $R$ a unital ring.  Then the group ring $RG$ is regular if and only if $R$ is regular, $G$ is locally finite and the order of each finite subgroup of $G$ is invertible in $R$.
\end{Thm}

Our main result extends Connell's theorem to ample groupoids.  




We shall need the following result going back to Weissglass~\cite{Weissglass}.  We sketch a proof for completeness.
\begin{Prop}[Weissglass]\label{p:inverse.semigroup.easy}
Let $R$ be a regular unital ring and let $S$ a locally finite inverse semigroup such that if a prime $p$ is the order of a cyclic subgroup of $S$, then $p$ is invertible in $R$.   Then $RS$ is regular.
\end{Prop}
\begin{proof}
Since the class of regular rings is closed under direct limits by Proposition~\ref{p:vnr}, we may assume that $S$ is a finite inverse semigroup.   Then it is well known that $RS$ is isomorphic to a finite direct product of matrix algebras over the group rings of its maximal subgroups (cf.~\cite[Theorem~4.6]{mobius2} where the assumption that $R$ is commutative is superfluous).  Our assumptions imply that the order of any maximal subgroup $G$ of $S$ is invertible in $R$ by Cauchy's theorem,  and so $RG$ is regular by Theorem~\ref{t:connell}.  Thus $RS$ is regular by Proposition~\ref{p:vnr}.
\end{proof}

\section{von Neumann regularity of Steinberg algebras}

We begin by establishing the necessity of the conditions in Theorem~A.

\begin{Prop}\label{p:R.is.regular}
Suppose that $R\mathscr G$ is regular with $\mathscr G$ ample and $R$ a unital ring.  Then $R$ is regular.
\end{Prop}
\begin{proof}
Let $r\in R$ and let $U$ be any nonempty compact open subset of $\mathscr G\skel 0$.  Let $f=r1_U$.  Then there exists $g\in R\mathscr G$ with $fgf=f$.  Let $x\in U$.  Then $r=f(x)=fgf(x)=rg(x)r$, as $x=\alpha\beta\gamma$ with $\alpha,\gamma\in U$ implies $\alpha=x=\gamma$, and hence $\beta=x$.  We conclude that $R$ is regular.
\end{proof}

Let us say that an ample groupoid $\mathscr G$ is \emph{uniformly bounded} if there is an integer $M\geq 1$ with $|\ran^{-1}(x)|\leq M$ for all $x\in \mathscr G\skel 0$ or, equivalently, $|\sour^{-1}(x)|\leq M$ for all $x\in \mathscr G\skel 0$.  Note that any open or closed subgroupoid of a uniformly bounded ample groupoid is uniformly bounded, and so, in particular, a uniformly bounded groupoid has finite isotropy groups.  Let us say that $\mathscr G$ is \emph{approximately uniformly bounded} if it is a directed union of uniformly bounded open subgroupoids.  An open or closed subgroupoid of an approximately uniformly bounded ample groupoid is again approximately uniformly bounded, and hence the isotropy groups of such a groupoid are locally finite. Let us say that $\mathscr G$ is \emph{approximately quasi-compact} if $\mathscr G$ is a directed union of quasi-compact open subgroupoids.   It is not difficult to verify that a principal second countable Hausdorff ample groupoid is approximately quasi-compact if and only if it is an approximately finite groupoid in the sense of~\cite[Section~11.5]{NylandOrtega19} (i.e., an approximately finite \'etale equivalence relation in the sense of~\cite{GPS}).
  Note that if $\mathscr G$ is a quasi-compact ample groupoid, then $\mathscr G\skel 0$ is compact being the image of $\mathscr G$ under $\ran$.

The following proposition will be used without comment.

\begin{Prop}\label{p:almost.compact.closed.under.dl}
Let $\mathscr G$ be an ample groupoid and suppose that it is a directed union of open approximately quasi-compact subgroupoids.  Then $\mathscr G$ is approximately quasi-compact.
\end{Prop}
\begin{proof}
Suppose that $\mathscr G=\bigcup_{\alpha\in D}\mathscr G_{\alpha}$ with the $\mathscr G_{\alpha}$ a directed family of open approximately quasi-compact  subgroupoids.  Write $\mathscr G_{\alpha}=\bigcup_{\beta\in I_{\alpha}} \mathscr G_{\alpha,\beta}$ with the $\mathscr G_{\alpha,\beta}$ a directed family of open quasi-compact subgroupoids of $\mathscr G_{\alpha}$.  Then each $\mathscr G_{\alpha,\beta}$ is open and quasi-compact in $\mathscr G$ and $\mathscr G=\bigcup_{\alpha\in D}\bigcup_{\beta\in I_{\alpha}}\mathscr G_{\alpha,\beta}$.  It remains to show that this family of open quasi-compact subgroupoids is directed.  Given $\mathscr G_{\alpha,\beta}$ and $\mathscr G_{\alpha',\beta'}$, there exists $\alpha_0$ with $\mathscr G_{\alpha},\mathscr G_{\alpha'}\subseteq \mathscr G_{\alpha_0}$.  Since $\mathscr G_{\alpha,\beta}\cup \mathscr G_{\alpha',\beta'}$ is quasi-compact and $\mathscr G_{\alpha_0}=\bigcup_{\beta_0\in I_{\alpha_0}}\mathscr G_{\alpha_0,\beta_0}$, which is a directed union of open subgroupoids, we deduce that $\mathscr G_{\alpha,\beta},\mathscr G_{\alpha',\beta'}\subseteq \mathscr G_{\alpha_0,\beta_0}$ for some $\beta_0\in I_{\alpha_0}$.   This establishes directedness.
\end{proof}

Our next observation is trivial.

\begin{Lemma}\label{l:uniform.bound}
Let $\mathscr G$ be an ample groupoid and $\mathscr H$ an open subgroupoid. If $\mathscr H$ is contained in a quasi-compact subspace of $\mathscr G$,   then $\mathscr H$ is uniformly bounded.  In particular, quasi-compact ample groupoids are uniformly bounded.
\end{Lemma}
\begin{proof}
Let $C$ be a quasi-compact subspace of $\mathscr G$ containing $\mathscr H$.  We may cover $C$ by finitely many compact open bisections $U_1\cup\cdots \cup U_n$ of $\mathscr G$.  Then  we have $|\ran\inv(x)\cap \mathscr H|\leq n$ for all $x\in \mathscr H\skel 0$ as $|\ran\inv(x)\cap U_i|\leq 1$ for $i=1,\ldots, n$ and $\mathscr H\subseteq U_1\cup\cdots \cup U_n$.
\end{proof}

The following result shows, amongst other things, that the classes of approximately quasi-compact and approximately uniformly bounded groupoids coincide.

\begin{Prop}\label{p:uniformly.bdd.vs.local.finite}
Let $\mathscr G$ be an ample groupoid.  Then the following are equivalent.
\begin{enumerate}
  \item $\mathscr G$ is approximately quasi-compact.
  \item $\mathscr G$ is approximately uniformly bounded.
  \item The inverse semigroup $\Gamma_c(\mathscr G)$ is locally finite.
  \item There is a locally finite inverse semigroup $S\subseteq \Gamma_c(\mathscr G)$ with $\bigcup S=\mathscr G$.
\end{enumerate}
\end{Prop}
\begin{proof}
Lemma~\ref{l:uniform.bound} shows that (1) implies (2).  Trivially, (3) implies (4).  If $T$ is any inverse semigroup of compact open bisections of $\mathscr G$, then $\bigcup T$ is an open subgroupoid.  Suppose that $S$ is a locally finite inverse semigroup of compact open bisections with $\bigcup S=\mathscr G$.  Then $S$ is the directed union of its finitely generated inverse subsemigroups, each of which is finite, and hence $\mathscr G$ is the directed union of the groupoids of the form $\mathscr H=\bigcup T$ with $T$ a finite inverse subsemigroup.  But since $T$ is finite and each element of $T$ is compact open, it follows  that $\mathscr H$ is quasi-compact and open, and so (4) implies (1).  It remains to show that (2) implies (3).

Suppose that (2) holds.  If $\mathscr G$ is the directed union of uniformly bounded open subgroupoids $\mathscr G_{\alpha}$ with $\alpha \in D$, then $\Gamma_c(\mathscr G)$ is the directed union of the $\Gamma_c(\mathscr G_{\alpha})$ since any compact set belongs to some $\mathscr G_{\alpha}$. Hence any finitely generated subsemigroup of $\Gamma_c(\mathscr G)$ is contained in some $\Gamma_c(\mathscr G_{\alpha})$.   Thus we may assume without loss of generality that $\mathscr G$ is uniformly bounded.  Suppose that $|\ran(x)|\leq n$ for all $x\in \mathscr G\skel 0$.  Then, for each $x\in \mathscr G\skel 0$, we have that $|\mathscr G^x_x|\cdot |\mathcal O_x|=|\ran\inv (x)|\leq n$.  It follows that $|\mathscr G|_{\mathcal O_x}| = |\mathcal O_x|^2\cdot |\mathscr G_x^x|\leq n^2$, and hence $\mathscr G|_{\mathcal O_x}$ is a finite closed subgroupoid,  and $\Gamma_c(\mathscr G|_{\mathcal O_x})$ is a finite inverse semigroup of cardinality at most $2^{n^2}$.

We have an injective homomorphism $\psi\colon \Gamma_c(\mathscr G)\to \prod_{\mathcal O\in \mathscr G\skel 0/\mathscr G} \Gamma_c(\mathscr G|_{\mathcal O})$ given by $\psi(U)_{\mathcal O} = U\cap \mathscr G|_{\mathcal O}$.     We deduce from Theorem~\ref{t:Birkhoff} that $\Gamma_c(\mathscr G)$ is locally finite, yielding (3).
\end{proof}

\begin{Cor}\label{c:gpd.germs}
Let $S$ be an inverse semigroup acting on a locally compact totally disconnected Hausdorff space $X$ by partial homeomorphisms between compact open subsets.  Let $\mathscr G$ be the groupoid of germs of the action.  If $S$ is locally finite, then $\mathscr G$ is approximately quasi-compact.  The converse holds if the mapping $\psi\colon S\to \Gamma_c(\mathscr G)$ given by $\psi(s) = (s,X_{s^*s})$ is injective.
\end{Cor}
\begin{proof}
We have a homomorphism $\psi\colon S\to \Gamma_c(\mathscr G)$ given by $\psi(s)=(s,X_{s^*s})$.  Moreover, $\mathscr G=\bigcup \psi(S)$.  If $S$ is locally finite, then so is $\psi(S)$ and hence $\mathscr G$ is approximately quasi-compact by Proposition~\ref{p:uniformly.bdd.vs.local.finite}.  Conversely, if $\mathscr G$ is approximately quasi-compact and  $\psi$ is injective, then $\Gamma_c(\mathscr G)$ and all its subsemigroups are locally finite by Proposition~\ref{p:uniformly.bdd.vs.local.finite}, and hence $S\cong \psi(S)$ is locally finite.
\end{proof}

Now we prove conditions (2) and (3) of Theorem~A are necessary.

\begin{Prop}\label{p:main.necessity}
Let $R$ be a unital ring and $\mathscr G$ an ample groupoid such that $R\mathscr G$ is von Neumann regular. Then $\mathscr G$ is approximately quasicompact and the order of any finite subgroup of an isotropy group of $\mathscr G$ is invertible in $R$.
\end{Prop}
\begin{proof}
First note that if $U$ is a compact open subset of $\mathscr G\skel 0$, then $\mathscr G|_U$ is an open subgroupoid with compact unit space and $R\mathscr G|_U = 1_UR\mathscr G1_U$. Hence if $R\mathscr G$ is regular, then so is $R\mathscr G|_U$  by Proposition~\ref{p:vnr}.  Since $\mathscr G$ is the directed union of  open subgroupoids of the form $\mathscr G|_U$, we may assume without loss of generality that $\mathscr G\skel 0$ is compact, that is, $R\mathscr G$ has the identity element $1_{\mathscr G\skel 0}$.

By Proposition~\ref{p:uniformly.bdd.vs.local.finite},  it suffices to show that $\mathscr G$ is approximately uniformly bounded, and that the order of any finite subgroup of an isotropy group is invertible in $R$.   If $S\leq \Gamma_c(\mathscr G)$ is an inverse submonoid, then $\bigcup S$ is an open subgroupoid of $\mathscr G$ containing $\mathscr G\skel 0$.  Moreover, since $\Gamma_c(\mathscr G)$ is the directed union of its finitely generated inverse submonoids and $\mathscr G=\bigcup \Gamma_c(\mathscr G)$, it suffices to show that if $S\leq \Gamma_c(\mathscr G)$ is a finitely generated inverse submonoid, then $\mathscr H=\bigcup S$ is uniformly bounded and the order of any isotropy group of $\mathscr H$ is invertible in $R$.  We proceed by adapting the idea of Connell for the case of groups~\cite{connell}.

Let $I_S$ be the left ideal of $R\mathscr G$ generated by the elements $1_U-1_{UU\inv}$ with $U\in S$.  Let $X$ be a finite generating set of $S$ as a semigroup which is closed under taking inverses.  We prove that $I_S$ is generated as a left ideal by the finite set $Y$ of elements of the form $1_V-1_{VV\inv}$ with $V\in X$.  We proceed by induction on word length, i.e., the minimum length of an expression of an element $U$ of $S$ as a product of elements of $X$, to show that $1_U-1_{UU\inv}\in K\mathscr G\cdot Y$, with the case of length $1$ being trivial.  If $U=WV$ with $V\in X$ and $W\in S$ of length one shorter than $U$, then $W\inv$ also has length one shorter than $U$ (since $X$ is closed under inversion), and so
\[1_U-1_{UU\inv} = 1_W(1_V-1_{VV\inv})- 1_{WVV\inv}(1_{W\inv}-1_{W\inv W})\in K\mathscr G\cdot Y\] by induction, as required.

Since $R\mathscr G$ is a regular ring with identity and $I_S$ is a finitely generated left ideal, there is an idempotent $e\in R\mathscr G$ with $I_S=R\mathscr Ge$ by~\cite[Theorem~1.1]{Goodearlreg}.  Let $f=1_{\mathscr G\skel 0}-e$.  Then $f$ is an idempotent and $I_Sf=R\mathscr Ge(1-e)=0$, that is, $1_Uf=1_{UU\inv}f$ for all $U\in S$.  If $f=\sum_{i=1}^rc_i1_{U_i}$ with $U_i$ compact open bisections, then $f$ can be nonzero on at most $r$ elements from the same $\sour$-fiber, a fact that we shall soon use.

Let $\gamma\in \mathscr H$ and $\beta\in \mathscr G$ with $\sour(\gamma)=\ran(\beta)$.  We claim that $f(\gamma\beta)=f(\beta)$.  Indeed,    let $U\in S$ with $\gamma\in U$.  Then $(1_Uf)(\gamma\beta) = f(\beta)$ and $(1_{UU\inv})f(\gamma\beta)=f(\gamma\beta)$.  We conclude that $f(\gamma\beta)=f(\beta)$ for all $\gamma\in \mathscr H\cap \sour^{-1}(\ran(\beta))$.

Let $\ran_*\colon R\mathscr G\to R\mathscr G\skel 0$ be the left $R$-module homomorphism induced by the local homeomorphism $\ran$.  So \[\ran_*(f)(x)=\sum_{\ran(\gamma)=x}f(\gamma).\]  Note that if $U\in \Gamma_c(\mathscr G)$, then $\ran_*(1_U) =1_{UU\inv}$.   It is well known and easy to check that $\ran_*$ is a left $R\mathscr G$-module homomorphism where $R\mathscr G$ acts on $R\mathscr G\skel 0$ via the rule \[f_1 \cdot g_1(x) = \sum_{\ran(\gamma)=x}f_1(\gamma)g_1(\sour(\gamma))\] for $f_1\in R\mathscr G$ and $g_1\in R\mathscr G\skel 0$.  Notice that if $U\in \Gamma_c(\mathscr G)$ and $V\subseteq \mathscr G\skel 0$ is compact open, then $1_U\cdot 1_V = 1_{UVU\inv}$.

If $U\in S$, then $\ran_*(1_U-1_{UU\inv}) =0$, and so $I_S\subseteq \ker \ran_*$.  Therefore, $\ran_*(f) = \ran_*(1_{\mathscr G\skel 0}-e) = \ran_*(1_{\mathscr G\skel 0}) = 1_{\mathscr G\skel 0}$, as $e\in I_S$.  Let $x\in \mathscr G\skel 0$.  Then
\begin{equation}\label{eq:augment}
1=\ran_*(f)(x) = \sum_{\ran(\beta)=x}f(\beta)
\end{equation}
 and so we can find $\beta\in \mathscr G$ with $\ran(\beta)=x$ and $f(\beta)\neq 0$.  Then, for every $\gamma\in \mathscr H\cap \sour^{-1}(x)$, we have $f(\gamma\beta) =f(\beta)\neq 0$.  Since $f$ is nonzero on at most $r$ elements with source $\sour(\beta)$, we deduce that $|\sour\inv(x)\cap \mathscr H|< r+1$.    Finally, the group $\mathscr H_x^x$ acts freely on the left of $\ran\inv (x)$, and so if we choose a transversal $T$ for this action and use $f(\gamma\beta)=f(\beta)$ for $\gamma\in \mathscr H_x^x$, we obtain from \eqref{eq:augment}
\[1=|\mathscr H_x^x|\cdot \sum_{t\in T} f(t),\] and so $|\mathscr H_x^x|$ is invertible in $R$.  This completes the proof.
\end{proof}

We now prove the sufficiency of the conditions in~Theorem~A.

\begin{Prop}\label{p:group.bundle}
Let $\mathscr G$ be a quasi-compact ample group bundle and $R$ a regular unital ring such that the order of any cyclic subgroup of an  isotropy group of $\mathscr G$ is invertible in $R$.  Then $R\mathscr G$ is regular.
\end{Prop}
\begin{proof}
We know that $S=\Gamma_c(\mathscr G)$ is locally finite by Proposition~\ref{p:uniformly.bdd.vs.local.finite}. We show that $RS$ is regular. It will then follow that its quotient $R\mathscr G$ is regular by Proposition~\ref{p:vnr}.    By Proposition~\ref{p:inverse.semigroup.easy} it suffices to show that if a prime $p$ is the order of a cyclic subgroup of $S$, then $p$ is invertible in $R$.

Let $U\subseteq \mathscr G\skel 0$ be the identity of a maximal subgroup $H$, and note that $H$ consists of those $V\in S$ with $\sour(V)=U=\ran(V)$.  If $V\in H$ and $x\in U$, then denote by $v_x$ the unique element of $V$ with $\ran(v_x)=x$. Since $\mathscr G$ consists of isotropy, $\sour(v_x)=x$ as well.  Then we have an injective group homomorphism $\psi\colon H\to \prod_{x\in U} \mathscr G_x^x$ given by $\psi(V)_x=v_x$.      Note that if $V\in H$ has order $p$, then  $\psi(V)$ has order $p$,  and hence $v_x^p=x$ for all $x\in U$, and some $v_x$ has order $p$.  Therefore, $p$ is invertible in $R$ by assumption.
\end{proof}

We now consider a class of quasi-compact ample groupoids whose members are building blocks of more complicated ones.   The remaining proofs are very much inspired by~\cite[Lemma~3.4]{GPS}.

\begin{Lemma}\label{l:constant.orbit.size}
Let $\mathscr G$ be a quasi-compact ample groupoid and $R$ a regular unital ring such that the order of any cyclic subgroup of an isotropy group of $\mathscr G$ is invertible in $R$.  Suppose that all orbits of $\mathscr G$ have the same cardinality.    Then $R\mathscr G$ is regular.
\end{Lemma}
\begin{proof}
It follows from Lemma~\ref{l:uniform.bound} that the size of each orbit is finite, say that this common size is $n$.  Choose, for each $x\in \mathscr G\skel 0$, elements $\gamma_{x,1},\ldots, \gamma_{x,n}$ with $\sour(\gamma_{x,i})=x$ and $\mathcal O_x=\{\ran(\gamma_{x,1}),\ldots, \ran(\gamma_{x,n})\}$.  We may assume without loss of generality that $\gamma_{x,1}=x$ (the identity at $x$).  Choose for each $i$ a compact open bisection $U^x_{i,1}$ containing $\gamma_{x,i}$.  We may assume that $U^x_{1,1}\subseteq \mathscr G\skel 0$.  Since $\mathscr G\skel 0$ is Hausdorff with a basis of compact open sets and   $\ran(\gamma_{x,1}),\ldots, \ran(\gamma_{x_n})$ are pairwise distinct, we can find pairwise disjoint compact open neighborhoods $V_1,\ldots, V_n$ of $\ran(\gamma_{x,1}),\ldots, \ran(\gamma_{x,n})$, respectively, in $\mathscr G\skel 0$.  Replacing $U^x_{i,1}$ by $V_iU^x_{i,1}$, we may assume that the sets $U^x_i=\ran(U^x_{i,1})$ are pairwise disjoint.  Also, putting $U=\bigcap_{i=1}^n \sour(U^x_{i,1})$ and replacing $U^x_{i,1}$ by $U^x_{i,1}U$, we may assume that $\sour(U^x_{i,1})=\sour(U^x_{1,1})=U^x_1$ for all $i$.  Put $U^x_{i,j} = U^x_{i,1}(U^x_{j,1})\inv$.  Note that $\sour(U^x_{i,j})=U^x_j$ and $\ran(U^x_{i,j})=U^x_i$ for all $i,j$.  Also note that $U^x_{i,i} = U^x_i$ for $i=1,\ldots, n$.   Let $U^x=\bigcup_{i=1}^n U^x_i$.  Then $U^x$ is compact open and we claim that it is invariant.  For suppose that $y\in U^x_j$.  Then since $U^x_j = \sour(U^x_{i,j})$  and $U^x_i = \ran(U^x_{i,j})$ for all $i=1,\ldots, n$, we can find elements of the orbit of $y$ in each of $U^x_1,\ldots, U^x_n$.  Since these sets are pairwise disjoint, this yields $n$ elements in the orbit of $y$ and hence the whole orbit by assumption.  Thus each $U^x$ is invariant and each $U^x_i$ contains exactly one element from each orbit of $U^x$.  Also note that
\begin{equation}\label{eq:matrix.units}
U^x_{i,j}U^x_{k,\ell}=\begin{cases}U^x_{i,\ell}, & \text{if}\ j=k\\ \emptyset, & \text{otherwise.}\end{cases}
\end{equation}

Another observation is that the isotropy subgroupoid $\Iso(\mathscr G)$ of $\mathscr G$ is clopen, and hence quasi-compact.  It is closed since $\mathscr G\skel 0$ is Hausdorff and it is the equalizer of $\sour,\ran$.  It is open because if $\gamma\colon x\to x$ is isotropy and $V$ is a compact open bisection containing $\gamma$, then $U^x_1VU^x_1$ is a compact open bisection containing $\gamma$ and consisting of isotropy as $x\in U^x_1$, and $U^x_1$ contains exactly only one element from each orbit of the invariant  compact open set $U^x$, whence $\sour(\alpha),\ran(\alpha)\in U^x_1$ implies $\sour(\alpha)=\ran(\alpha)$.

Since $\mathscr G\skel 0$ is compact and the $U^x$ cover $\mathscr G\skel 0$, we can find finitely may points $x_1,\ldots, x_r$ such that $\mathscr G\skel 0=U^{x_1}\cup\cdots \cup U^{x_r}$. Without loss of generality, we may assume that $U^{x_i}\nsubseteq \bigcup_{j\neq i}U^{x_j}$ for all $i$.   Put $V_1=U^{x_1},V_2 = U^{x_2}\setminus V_1,\ldots, V_r=U^{x_r}\setminus \bigcup_{i=1}^{r-1}V_i$.  Then $V_1,\ldots, V_r$ is  partition of $\mathscr G\skel 0$ into compact open invariant sets.  In particular note that $V_iU=UV_i$ for any compact open bisection $U$.  It follows that if $\mathscr G_i = \mathscr G|_{V_i}$, then $R\mathscr G\cong \prod_{k=1}^r R\mathscr G_k$ and $R\mathscr G_k = 1_{V_k}R\mathscr G1_{V_k}$.

Note that since $\mathscr G_k=\ran\inv(V_k)$ is clopen in $\mathscr G$ it is quasi-compact.  Also $\Iso(\mathscr G_k)=\Iso(\mathscr G)\cap \mathscr G_k$ is clopen and hence quasi-compact.
Put $U^k_{i,j} = V_kU^{x_k}_{i,j}=U^{x_k}_{i,j}V_k$; this is a compact open bisection in $\mathscr G_k$.  Also note that $U^k_{1,1}\subseteq U^{x_k}_1$ does not contain any two points in the same orbit by the previous discussion.  Hence $U^k_{1,1}\mathscr G_kU^k_{1,1} = \mathscr G_k|_{U^k_{1,1}}=\Iso(\mathscr G_k)|_{U^k_{1,1}}=\Iso(\mathscr G_k)\cap \ran\inv(U^k_{1,1})$ is clopen (as $U^k_{1,1}$ is a clopen subspace of $\mathscr G\skel 0$) and thus a quasi-compact group bundle.
Note that $V_k=V_kU^{x_k}=\bigcup_{i=1}^n U^k_{i,i}$ and hence it follows from \eqref{eq:matrix.units} that the $1_{U^k_{i,j}}$ form a set of matrix units for $R\mathscr G_k$. Thus $R\mathscr G_k\cong M_n(1_{U^k_{1,1}}R\mathscr G_k1_{U^k_{1,1}})=M_n(R\Iso(\mathscr G_k)|_{U^k_{1,1}})$.  We deduce that $R\mathscr G$ is regular from Propositions~\ref{p:vnr} and~\ref{p:group.bundle}.
\end{proof}

We need one last lemma before proving our next main result.

\begin{Lemma}\label{l:lower.semi}
Let $\mathscr G$ be an \'etale groupoid and $k\geq 1$ be a natural number.  Then the set $O_k$ of elements $x\in \mathscr G\skel 0$ whose orbit has at least $k$ elements is an open invariant subset.
\end{Lemma}
\begin{proof}
Clearly $O_k$ is invariant; we claim that it is open.  Suppose that $x\in O_k$.  Let $x=y_1,y_2\ldots, y_k$ be $k$ distinct elements in the orbit of $x$.  As in the proof of Lemma~\ref{l:constant.orbit.size}, we can find open bisections $U_1,\ldots, U_k$ with $x\in \sour(U_i)$, $y_i\in \ran(U_i)$, for $i=1,\ldots, k$, and $\ran(U_1),\ldots, \ran(U_k)$ pairwise disjoint (using that $\mathscr G\skel 0$ is Hausdorff).  Then $U=\bigcap_{i=1}^k\sour(U_i)$ is a neighborhood of $x$ and if $y\in U$, then each of $\ran(U_1),\ldots,\ran(U_k)$ contains an element of the orbit of $y$.  Since these sets are pairwise disjoint, we conclude that $y\in O_k$.
\end{proof}

We can now prove the sufficiency of the conditions in Theorem~A.

\begin{Prop}\label{p:suff}
Suppose that $R$ is a regular unital ring and $\mathscr G$ is an approximately quasi-compact ample groupoid such that the order any finite subgroup of an isotropy group is invertible in $R$.  Then $R\mathscr G$ is regular.
\end{Prop}
\begin{proof}
Since an approximately quasi-compact ample groupoid is a directed union of quasi-compact open subgroupoids, it suffices to prove the result for quasi-compact groupoids as regular rings are closed under direct limits.  So assume that $\mathscr G$ is quasi-compact and hence uniformly bounded by Lemma~\ref{l:uniform.bound}.  In particular, the cardinalities of the orbits of $\mathscr G$ are uniformly bounded.

 Let $d_1<d_2<\cdots<d_m$ be the distinct orbit sizes.  We prove the result by induction on $m$.  If $m=1$, then all orbits of $\mathscr G$ have the same size and the result follows from Lemma~\ref{l:constant.orbit.size}.   Let $U=\{x\in \mathscr G\skel 0\mid |\mathcal O_x|=d_m\}$.  Then $U$ is open and invariant by Lemma~\ref{l:lower.semi} once we note that $U=O_{d_m}$ in the notation of that lemma.   Let $X=\mathscr G\skel 0\setminus U$.  Then $I=R\mathscr G|_U$ is an ideal of $R\mathscr G$ and $R\mathscr G/I\cong R\mathscr G|_X$.  Note that $\mathscr G|_U=\ran\inv (U)$ is open and $\mathscr G|_X=\ran\inv (X)$ is closed and hence quasi-compact.  By construction, the orbits sizes in $\mathscr G|_X$ are $d_1<d_2<\cdots<d_{m-1}$, and so by induction $R\mathscr G|_X$ is regular.     If we can show that $R\mathscr G|_U$ is regular, then Proposition~\ref{p:vnr} will imply that $R\mathscr G$ is regular.

Let $V\subseteq U$ be a compact open set.  Then $V$ is clopen in $\mathscr G\skel 0$, and hence $\ran\inv(V)\subseteq \mathscr G$ is clopen, whence quasi-compact and open.  Thus $W=\sour(\ran\inv(V))$ is a compact open invariant subset of $U$ containing $V$.  Hence the compact open invariant subsets of $U$ are cofinal among compact open subsets of $U$.  Moreover, if $W$ is a compact open invariant subspace of $U$, then  $W$ is clopen in $\mathscr G\skel 0$ (as the latter is Hausdorff), and so $\mathscr G|_W=\ran\inv(W)$ is a clopen subgroupoid of $\mathscr G$, hence quasi-compact, with all orbits of size $d_m$ (as $W\subseteq U$ is invariant).  Hence $R\mathscr G|_W$ is regular by Lemma~\ref{l:constant.orbit.size}.  Since $R\mathscr G|_U$ is the direct limit of the $R\mathscr G|_W$ where $W$ runs over the compact open invariant subsets of $U$, we conclude that $R\mathscr G|_U$ is regular from Proposition~\ref{p:vnr}.  This completes the proof.
\end{proof}

\begin{Rmk}
If one follows the proofs carefully, it is not difficult to see that if $K$ is a field and $\mathscr G$ is an approximately quasi-compact ample groupoid such that no isotropy element has order divisible by the characteristic of $K$, then $K\mathscr G$ is a directed union of finite dimensional unital semisimple algebras, and hence if $K$ is algebraically closed, then $K\mathscr G$ is ultramatricial.  This just uses that the class under consideration is closed under matrix amplification and if an algebra has a finite ideal filtration with factors of this form, then the algebra has this form.
\end{Rmk}

We may now prove Theorem~A.

\begin{proof}[Proof of Theorem~A]
Necessity of (1)--(3) follows from Propositions~\ref{p:R.is.regular} and \ref{p:main.necessity}.  Sufficiency of these conditions follows from Proposition~\ref{p:suff}.
\end{proof}

\section{Graded von Neumann regularity of Steinberg algebras}
\subsection{Graded regularity}

In this section we characterize when ample groupoid algebras are graded von Neumann regular.

If $H$ is a group with identity $e$, then a ring $\Lambda$ is \emph{$H$-graded} if $\Lambda$ has a direct sum decomposition $\Lambda=\bigoplus_{h\in H}\Lambda_h$, where each $\Lambda_h$ is an additive subgroup of $\Lambda$ such that $\Lambda_h\Lambda_k \subseteq \Lambda_{hk}$ for all $h,k\in H$. The summand $\Lambda_h$ is called the \emph{homogeneous component} of degree $h$. The homogeneous component $\Lambda_e$ is a subring of $\Lambda$.  An element is \emph{homogenous} if it belongs to a homogeneous component.

 A graded ring $\Lambda$ is \emph{graded (von Neumann) regular} if for every homogeneous element $x\in \Lambda$ there exists $y\in \Lambda$ such that $x=xyx$.

Let $\mathscr{G}$ be an ample groupoid and let $H$ be a discrete group. A cocycle from $\mathscr{G}$ to $H$ is a map $c\colon\mathscr{G}\to H$ such that 
$c(\gamma\eta) =   c(\gamma)c(\eta)$ for all $(\gamma,\eta)\in \mathscr{G}\skel 2$. A continuous cocycle $c\colon\mathscr{G}\to H$ induces an
an $H$-grading $\{(R\mathscr{G})_h\}_h$ of $R\mathscr{G}$.  This is done in detail in~\cite{operatorguys1} for Hausdorff groupoids. The details for non-Hausdorff groupoids are similar, but we present them for completeness.    Call a compact open bisection $U\in \Gamma_c(\mathscr G)$ homogenous of degree $h\in H$ if $U\subseteq c\inv(h)$ for some $h\in H$.   Then $(R\mathscr G)_h$ is the $R$-span of the characteristic functions of homogeneous compact open bisections of degree $h$. Note that each $c\inv(h)$ with $h\in H$ is clopen, and so if $U$ is any compact open bisection, then $U=\bigcup_{h\in H}(U\cap c\inv(h))$ and the $U\cap c\inv(h)$ are pairwise disjoint homogeneous compact open bisections. By compactness of $U$, only finitely many of these intersections, say corresponding to $h_1,\ldots, h_n$,  are nonempty, and so $1_U = \sum_{i=1}^n 1_{(c\inv (h_i)\cap U)}$, showing that $\sum_{h\in H}(R\mathscr G)_h=R\mathscr G$.  The sum is direct since the $c\inv(h)$ are pairwise disjoint and elements of $(R\mathscr G)_h$ are supported on $c\inv (h)$.  One can, in fact, show that $(R\mathscr G)_h$ consists of those elements of $R\mathscr G$ supported on $c\inv(h)$, but we shall never need this description.  This argument also shows that the homogeneous compact open bisections form a basis for the topology of $\mathscr G$.

 Notice that $\mathscr{G}^{(0)}\subseteq c^{-1}(e)=\mathscr G_e$ and that $\mathscr G_e$ is a clopen subgroupoid and hence is ample. Moreover, $(R\mathscr{G})_e =R\mathscr{G}_e$ by definition.  From now on, we write $R\mathscr G_h$ instead of the more cumbersome $(R\mathscr G)_h$.

The following is~\cite[Definition 2.9]{Lannstrom21}.
	Let $\Lambda=\bigoplus_{h\in H}\Lambda_h$ be an $H$-graded $K$-algebra. If, for every $h\in H$ and $r\in \Lambda_h$, there exist $\epsilon_h(r)\in \Lambda_h\Lambda_{h^{-1}}$ and $\epsilon_h^{\prime}(r)\in \Lambda_{h^{-1}}\Lambda_h$ such that $\epsilon_g(r)r = r = r\epsilon_g^{\prime}(r)$, then $\Lambda$ is called \emph{nearly epsilon-strongly $H$-graded}.  It is shown in~\cite{Lannstrom21} that being nearly epsilon-strongly $H$-graded is a necessary condition for graded regularity.

\begin{Prop}\label{e-strong}
	Let $\mathscr{G}$ be an ample groupoid with a continuous cocycle $c\colon\mathscr{G}\to H$ and let $R$ be a ring. Then $R\mathscr{G}$ is nearly epsilon-strongly $H$-graded.
\end{Prop}
\begin{proof}
	Let $\mathscr{G} = \bigoplus_{h\in H} R\mathscr{G}_h$  be the $H$-grading induced by the continuous cocycle $c\colon\mathscr{G}\to H$ as above.
	
	Fix any $h\in H$ and let $U\subseteq c^{-1}(h)$ be a compact open bisection. Define $\epsilon_h(1_U) =1_{UU^{-1}}$. Then
	\[\epsilon_h(1_U) =1_{UU^{-1}}=1_{U}1_{U^{-1}}\in R\mathscr{G}_h \, R\mathscr{G}_{h^{-1}},\]
	and
	\[\epsilon_h(1_U)1_U =1_{UU^{-1}}1_U =1_{UU^{-1}U} =1_U.\]
	Similarly, defining  $\epsilon^{\prime}_h(1_U) =1_{U^{-1}U}$ yields that $\epsilon^{\prime}_h(1_U) \in R\mathscr{G}_{h^{-1}} \, R\mathscr{G}_{h}$ and $ 1_U\epsilon^{\prime}_h(1_U)=1_U$. 	
	
	Now consider an arbitrary $f\in R\mathscr{G}_h$. Then there are compact open bisections $U_1,U_2,\ldots,U_n$, with each $U_i\subseteq c^{-1}(h)$, and $r_1,r_2,\cdots,r_n\in R$ such that $f =\sum_{i=1}^n r_i1_{U_i}$. In particular, $\epsilon_h(1_{U_i}) \in R\mathscr{G}_h \, R\mathscr{G}_{h^{-1}}$ for each $i$. Put $U=\bigcup_{i=1}^n U_iU_i^{-1}$. Then $U$ is a compact open set in $\mathscr{G}^{(0)}$. Put $\epsilon_h(f) = 1_U$. Since $U\subseteq \mathscr{G}^{(0)}$ and $\mathscr{G}^{(0)}$ is Hausdorff, it follows from the inclusion-exclusion principle that $1_U=\epsilon_h(f)$ is an integer combination of the $1_{U_iU_i\inv}$, and hence $\epsilon_h(f) \in R\mathscr{G}_h \, R\mathscr{G}_{h^{-1}}$.  Similarly, putting $U^{\prime} =\bigcup_{i=1}^n U_i^{-1}U_i$ and defining  $\epsilon^{\prime}_h(f) = 1_{U^{\prime}}$, we have that  $\epsilon^{\prime}_h(f) \in R\mathscr{G}_{h^{-1}} \, R\mathscr{G}_{h}$. It is straightforward from the construction that $\epsilon_h(f) f= f = f\epsilon^{\prime}_h(f)$. Hence, $R\mathscr{G}$ is nearly epsilon-strongly $H$-graded.
\end{proof}

\begin{Thm}\label{gr_regular}
	Let $\mathscr{G}$ be an ample groupoid, $R$ a unital ring, $H$ a discrete group and $c\colon\mathscr{G}\to H$ a continuous cocycle. Then $R\mathscr{G}$ is graded regular if and only if $R\mathscr{G}_e$ is regular.
\end{Thm}
\begin{proof}
	By Proposition~\ref{e-strong}, $R\mathscr{G}$ is nearly epsilon-strongly $H$-graded. Thus, by~\cite[Theorem 1.2]{Lannstrom21}, $R\mathscr{G}$ is graded regular if and only if $R\mathscr{G}_e$ is regular.
\end{proof}

Theorem~\ref{gr_regular} allows us to upgrade Theorem~A to a theorem about graded regularity.

\begin{Thm}\label{t:graded.vnr}
Let $\mathscr G$ be an ample groupoid, $R$ a unital ring, $H$ a discrete group and $c\colon \mathscr G\to H$ a continuous cocycle.  Then $R\mathscr G$ is graded von Neumann regular (with respect to the $H$-grading induced by $c$) if and only if the following conditions hold:
\begin{enumerate}
\item $R$ is regular
\item $\mathscr G_e$ is approximately quasi-compact
\item  the order of each finite subgroup of an isotropy group of $\mathscr G_e$ is invertible in $R$.
\end{enumerate}
\end{Thm}

\section{Applications}

Our first application is to prove Theorem~B.

\begin{proof}[Proof of Theorem~B]
The sufficiency of the conditions is proved in Proposition~\ref{p:inverse.semigroup.easy}.  For sufficiency,
suppose that $RS$ is regular.   Let  $\mathscr{G}_S$ denote the universal groupoid of $S$. Then $R\mathscr G_S\cong RS$ is regular, and so $R$ is regular, $\mathscr G_S$ is approximately quasi-compact and the order of each finite subgroup of an isotropy group of $\mathscr G_S$ is invertible in $R$.  Then $S$ is locally finite  by Corollary~\ref{c:gpd.germs} since the canonical mapping $\psi\colon S\to \Gamma_c(\mathscr G_S)$ given by $\psi(s)=(s,D_{s^*s})$ is injective.  A maximal subgroup $G_e$ of $S$ is isomorphic to the isotropy group at the principal character $x_e$, and hence the order of any finite subgroup of $G_e$ is invertible in $R$.
\end{proof}


Next we give an example to show that the necessary conditions for regularity in~\cite{AHL} are not sufficient to guarantee regularity of the algebra of a groupoid by giving examples of principal groupoids with nonregular algebras over any regular unital ring $R$.   Recall that if a discrete group $G$ acts on a locally compact, Hausdorff and totally disconnected space $X$ by homeomorphisms, then the groupoid of germs of this action is just the so-called transformation groupoid $\mathscr G$ with underlying space $G\times X$
(with the product topology) and unit space $\{1\}\times X$ (identified with $X$) with $\sour(g,x)=x$ and $\ran(g,x)=gx$.  Products are given by $(h,gx)(g,x)=(hg,x)$ and $(g,x)\inv = (g\inv ,gx)$.  The isotropy group of the transformation groupoid at $x\in X$ is just stabilizer of $x$ in $G$.  In particular, if the action of $G$ is free, then the transformation groupoid is principal.  It is well known that $K\mathscr G$ is isomorphic to the skew group ring $C_c(X,R)\rtimes G$ where $C_c(X,R)$ is the ring of compactly supported locally constant functions $X\to R$ and $G$ acts on this ring via $(gf)(x)=f(g\inv x)$.   The first part of the following result extends a result in~\cite{GPS}.

\begin{Thm}\label{t:trans.grpd}
Let $X$ be compact totally disconnected space and $G$ a discrete group acting on $X$.  Let $\mathscr G$ be the corresponding transformation groupoid.
\begin{enumerate}
  \item $\mathscr G$ is approximately quasi-compact if and only if $G$ is locally finite.
  \item If $R$ is a regular unital ring, then $R\mathscr G$ is regular if and only if $G$ is locally finite and the order of each finite subgroup of an isotropy group of the action is invertible in $R$.
\end{enumerate}
\end{Thm}
\begin{proof}
The first item follows from Corollary~\ref{c:gpd.germs} since the map $\psi\colon G\to \Gamma_c(\mathscr G)$ given by $\psi(g) = \{g\}\times X$ is injective.  The second item follows from the first and Theorem~A.
\end{proof}

For example, if we consider a minimal  (nonperiodic) action of $\mathbb Z$ on the Cantor set $X$ (e.g., a nonperiodic minimal shift space), then the associated transformation groupoid $\mathscr G$ is a principal groupoid (and hence trivially satisfies the necessary conditions in~\cite{AHL}), but $\mathbb C\mathscr G$ is not regular.

\subsection{Leavitt path algebras}
We follow the conventions and notation of~\cite{LeavittBook} for Leavitt path algebras. For the boundary path groupoid of a graph we use the terminology and conventions of~\cite{NylandOrtega19}.

A \emph{directed graph} $E=(E^0,E^1,r,s)$ consists of a set of vertices $E^0$, a set of edges $E^1$ and maps $r,s\colon E^1\to E^0$. A finite path in $E$ is a  finite sequence of edges $\mu= e_1e_2\cdots e_n$ such that $r(e_i)=s(e_{i+1})$.   The length of $\mu$ is $|\mu|=n$.  We also admit an empty path of length $0$ at each vertex that we often identify with the vertex. An infinite path is an infinite sequence of edges  $e_1e_2\cdots$ such that $r(e_i)=s(e_{i+1})$ for all $i\geq 1$. Let $E^*$ denote the set of all finite paths in $E$ and let $E^\infty$ denote the set of all infinite paths in $E$. A finite path $e_1e_2\cdots e_n$ is a closed path if $s(e_1)=r(e_n)$, and $E$ is acyclic if it has no nonempty closed paths. A vertex $v\in E^0$ is a sink if $|s^{-1}(v)|=0$, an infinite emitter if $|s^{-1}(v)|=\infty$ and a regular vertex otherwise. Let $E_{\text{reg}}$ denote all regular vertices, and call $E_{\text{sing}} = E^0\setminus E_{\text{reg}}$ the set of singular vertices.  The mappings $r,s$ extend to paths in the natural way.

The boundary path space of $E$ is the set  $\partial E=E^\infty \cup \{\mu\in E^* \mid  r(\mu)\in E_{\text{sing}}\}$. If $\mu\in E^*$, we let $Z(\mu) = \{\mu x\mid  r(\mu)=s(x), x\in \partial E\}$.  Sets of the form
\[ Z(\mu\setminus F) =  Z(\mu) \setminus \left(\cup_{e\in F} Z(\mu e)\right),\]
where $\mu\in E^*$ and finite $F\subseteq \{e\in E^1\mid  s(e) = r(\mu)\}$, form a basis of compact open sets for a locally compact Hausdorff topology on $\partial E$.

Let $\partial E^{\geq n} = \{\mu\in \partial E\mid |\mu|\geq n \}$. Then each $\partial E^{\geq n}$ is an open set. The shift map $\sigma\colon E^{\geq 1} \to \partial E$ is given by $\sigma(e_1e_2e_3\cdots) = e_2e_3\cdots$ for $e_1e_2e_3\cdots\in \partial E^{\geq 2}$ and $\sigma(e) =r(e)$ if $e\in E^1$. Let $\sigma^n\colon \partial E^{\geq n}\to \partial E$ denote the $n$-fold composition of $\sigma$, and let $\sigma^0=\mathrm{id}$. Then each $\sigma^n$ is a local homeomorphism between open subsets of $\partial E$.

As a set, the boundary path groupoid associated with $E$ is
\[\mathscr{G}_E = \{(x,m-n,y)\in\partial E\times \mathbb{Z}\times \partial E \mid  x\in \partial E^{\geq m}, y\in \partial E^{\geq n}, \sigma^m(x)=\sigma^n(y)\}.\]
The set of composable pairs is \[\mathscr{G}_E\skel 2 =\{((x,m,y),(y,n,z))\mid (x,m,y),(y,n,z)\in \mathscr{G}_E\}.\] Multiplication and  inverses are given by $(x,m,y)(y,n,z) = (x,m+n,z)$ and $(x,m,y)^{-1}=(y,-m,x)$, respectively. For $\mu,\nu\in E^*$ with $r(\mu)=r(\nu)$ such that $\sigma^{|\mu|}(Z(\mu\setminus F)) = \sigma^{|\nu|}(Z(\nu\setminus F))$ and finite $F\subseteq \{e\in E^1\mid  s(e) = r(\mu)\}$, let
\[Z(\mu,F,\nu) = \{(x,|\mu|-|\nu|,y) \in \mathscr{G}_E\mid  x\in Z(\mu\setminus F), y\in Z(\nu\setminus F)  \}.\]
Then the collection $\{Z(\mu,F,\nu)\}$, ranging over  $\mu,\nu\in E^*$ with $r(\mu)=r(\nu)$ and finite $F\subseteq \{e\in E^1\mid  s(e) = r(\mu)\}$, forms a basis of compact open bisections  for a locally compact Hausdorff topology on $\mathscr{G}_E$~\cite[Lemma 9.2]{NylandOrtega19}. In particular, $\mathscr{G}_E$ is an ample groupoid.  The unit space $\mathscr{G}_E^{(0)}$ is identified with $\partial E$.

All isotropy groups in $\mathscr{G}_E$ are (isomorphic to) subgroups of $\mathbb{Z}$ via $(x,m,x)\mapsto m$ for $x\in \mathscr G_E\skel 0$. The isotropy group at $x\in \partial E$ is nontrivial if and only if $x$ is eventually periodic, that is, $x = e_1e_2\cdots e_n\mu^\infty$ for some $n\in \mathbb{N}\cup\{0\}$ and closed path $\mu$ of positive length. Hence, $\mathscr{G}_E$ is principal if and only if $E$ is acyclic.

Let $K$ be a commutative ring with unit and let $L_K(E)$ denote the Leavitt path algebra of $E$ over $K$ (see~\cite{Leavitt1} for the definition of $L_K(E)$).  Then $L_K(E)\cong K\mathscr{G}_E$  with the isomorphism given by $\mu\nu^* \mapsto 1_{Z(\mu,\nu)}$~\cite[Proposition 4.3]{operatorguys1}.  In~\cite[Theorem 1]{AbramRangaswamy10} it is proved that $L_K(E)$ is regular if and only if $E$ acyclic. Using our groupoid characterization, we can give a short proof of this equivalence, and expand on it.

\begin{Thm}\label{Leavitt_reg}
	Let $E$ be a  directed graph  and $K$ a von Neumann regular commutative ring. The following are equivalent.
	\begin{enumerate}
		\item $L_K(E)$ is von Neumann regular.
		\item $E$ is acyclic.
		\item $\mathscr{G}_E$ is principal.
\item $\mathscr G_E$ is an approximately quasi-compact groupoid.
	\end{enumerate}
\end{Thm}
\begin{proof}
We already observed that acyclic is equivalent to principal.  If $L_K(E)\cong K\mathscr G_E$ is regular, then $\mathscr G_E$ has locally finite isotropy by Theorem~A and hence $E$ must be principal, as $\mathscr G_E$ has no infinite cyclic isotropy group.  Suppose now that $E$ is acyclic.  We show that $\mathscr G_E$ is approximately quasi-compact.

  Let $F$ be a finite subgraph of $E$.  Then consider the union $\mathscr G_E(F)$ of all compact open bisections $Z(\alpha,\emptyset,\beta)$  with $\alpha,\beta$ finite paths in $F$ such that $r(\alpha)=r(\beta)$.  Then $\mathscr G_E(F)$ is compact open because  there are only finitely many such paths $\alpha,\beta$ as $F$ is finite and acyclic.  It is a subgroupoid since if $\alpha,\beta,\gamma,\delta$ are finite paths in $F$ with $r(\alpha)=r(\beta)$ and $r(\gamma)=r(\delta)$, then
  \begin{equation}\label{eq:mult.graph}
  Z(\alpha,\emptyset,\beta)Z(\gamma,\emptyset,\delta)=\begin{cases}Z(\alpha\mu,\emptyset,\delta), & \text{if}\ \gamma=\beta\mu\\ Z(\alpha,\emptyset,\delta\nu), & \text{if}\ \beta=\gamma\nu\\ \emptyset, & \text{else.}\end{cases}
  \end{equation}
  Since $\mathscr G_E$ is the directed union of the $\mathscr G_E(F)$ as $F$ ranges over all finite subgraphs of $E$, we conclude that $\mathscr G_E$ is approximately quasi-compact.

  If $\mathscr G_E$ is approximately quasi-compact (and hence approximately uniformly bounded), then all its isotropy groups must be locally finite.  Since the isotropy groups of $\mathscr G_E$ are always subgroups of $\mathbb Z$, they must then be trivial, and so $L_K(E)\cong K\mathscr G_E$ is regular by Theorem~A.
\end{proof}

Next we show that we can recover~\cite[Theorem 10]{Hazrat14}, which states that the Leavitt path algebra of an arbitrary graph over a field is graded von Neumann regular. In fact, we can get away with having coefficients in a commutative regular ring.  Every Leavitt path algebra $L_K(E)$ over a commutative ring $K$ is $\mathbb{Z}$-graded, with the grading given by homogeneous components
\[L_K(E)_n = \mathrm{span}_K\{\alpha\beta^* \mid  \alpha,\beta\in E^*, |\alpha|-|\beta|=n \}.\]
The corresponding grading on $K\mathscr{G}_E$ is induced by the continuous cocycle $c\colon\mathscr{G}_E \to \mathbb{Z}$ defined by $(x,m,y)=m$.

\begin{Thm}[Hazrat]\label{Leavitt_graded_regular}
	Let $E$ be an arbitrary graph and $K$ a unital commutative von Neumann regular unital ring. Then the Leavitt path algebra $L_K(E)$ is a graded von Neumann regular ring.
\end{Thm}
\begin{proof}
To simplify notation, let $\mathscr{G} =\mathscr{G}_E$ be the boundary path groupoid of $E$.  Then $\mathscr G_0 =\{(\alpha x,0,\beta x)\in \mathscr G\mid |\alpha|=|\beta|\}$. In particular, the isotropy groups of $\mathscr G_0$ are trivial as $(x,m,x)\to m$ is an isomorphism of the isotropy group of $x$ in $\mathscr G$ with a subgroup of $\mathbb Z$.
For a finite subgraph $F$ of $E$ and $n\geq 0$, we let \[\mathscr G_0(F,n)=\bigcup_{\alpha,\beta\in F^*, \ran(\alpha)=\ran(\beta),|\alpha|=|\beta|\leq n} Z(\alpha,\emptyset,\beta).\]  Then $\mathscr G_0(F,n)$ is compact open, as there are only finitely many choices of $\alpha,\beta$ given $F$ and $n$, and is a subgroupoid by \eqref{eq:mult.graph}.  If $F\subseteq F'$ and $n\leq n'$, then $\mathscr G_0(F,n)\subseteq \mathscr G_0(F',n')$.  It follows that $\mathscr G_0$ is approximately quasi-compact.  We deduce now from Theorem~\ref{t:graded.vnr} that $L_K(E)\cong K\mathscr G_E$ is graded regular.
\end{proof}

\subsection{Algebraic partial skew group rings induced by topological partial actions}

We refer the reader to~\cite{ExelBook17} for a detailed treatment on partial actions and partial skew group rings.

Let $X$ be a locally compact Hausdorff totally disconnected space, let $G$ be a discrete group with identity element $e$, and let $K$ be a commutative ring. Let $\phi=(\{X_g\}_{g\in G}, \{\phi_g\}_{g\in G})$ be a topological partial action of $G$ on $X$. In particular, each $X_g$ is a clopen subset of $X$, and $\phi_g\colon X_{g^{-1}}\to X_g$ is a homeomorphism with inverse $\phi_g^{-1} = \phi_{g^{-1}}$. For any $g\in G$, let $C_c(X_g,K)$ denote the ring of all locally constant compactly supported functions $f\colon X_g\to K$. Since $X_g$ is open in $X$, we may view $C_c(X_g,K)$ as an ideal in $C_c(X,K)$. Let $D_g= C_c(X_{g},K)$ and define $\varphi_g\colon D_{g^{-1}}\to D_g$ by $\varphi_g(f)(x) = f\circ \phi_{g^{-1}}(x)$. Then $\varphi_g$ is an isomorphism of $K$-algebras and $\varphi =(\{D_g\}_{g\in G},\{\varphi_g\}_{g\in G})$ is an algebraic partial action of $G$ on $C_c(X,K)$. The partial skew group ring is the $K$-algebra with underlying $K$-vector space
\[ C_c(X,K)\rtimes_\varphi G = \bigoplus_{g\in G} D_g\delta_g\]
and multiplication defined by $(a\delta_g)(b\delta_h) = \varphi_g((\varphi_{g^{-1}}(a)b)\delta_{gh}$. The above direct sum decomposition is a $G$-grading on $C_c(X,K)\rtimes_\varphi G$ (by definition of the product) with the identity component given by $(C_c(X,K)\rtimes_\varphi G)_e = D_e\delta_e \cong C_c(X,K)$.

The transformation groupoid associated with the partial action $\phi$ is the set
\[\mathscr{G}=\{(x,g,y)\in X\times G \times X\mid \, y\in X_{g^{-1}} \text{ and } x=\phi_g(y)\}, \]
with the set of composable pairs given by
\[\mathscr{G}\skel 2=\{((x,g,y),(y,h,z))\mid  (x,g,y),(y,h,z)\in \mathscr{G} \}\]
and multiplication and inverses defined by
\[(x,g,y)(y,h,z) = (x,gh,z) \text{ and } (x,g,y)^{-1} = (y,g^{-1},x).\]
With the relative topology from the product topology on $X\times G\times X$, $\mathscr{G}$ is an ample Hausdorff groupoid. The unit space $\mathscr{G}^{(0)}$ of $\mathscr{G}$ is identified with $X$ (as topological spaces). The groupoid has a natural $G$-grading induced by the continuous cocycle $c\colon \mathscr{G}\to G$ given by $c(x,g,y) =g$.
Hence, $K\mathscr{G}$ is $G$-graded. Moreover,  $K\mathscr{G}$ is graded isomorphic to $C_c(X,K)\rtimes_\varphi G$. Note that $K\mathscr{G}_e = K\mathscr{G}^{(0)}\cong C_c(X,K)$.

\begin{Prop}\label{partial_regular_graded}
	Let $\phi=(\{X_g\}_{g\in G}, \{\phi_g\}_{g\in G})$ be a topological partial action of a discrete group $G$ on locally compact Hausdorff and totally disconnected space $X$. Then $C_c(X,K)\rtimes_\varphi G$ is graded regular.
\end{Prop}
\begin{proof}
	Let $\mathscr{G}$ denote the transformation groupoid associated with $\phi$, and let $e\in G$ denote the identity element. By Theorem~\ref{gr_regular}, $C_c(X,K)\rtimes_\varphi G\cong K\mathscr{G}$ is graded regular if and only if $K\mathscr{G}_e\cong C_c(X,K)$ is regular.  But $C_c(X,K)$ is trivially regular.  Indeed, if $f\in C_c(X,K)$, then we can write $f=\sum_{i=1}^n k_i1_{U_i}$ with the $U_i$ disjoint compact open sets (as $X$ is Hausdorff).  If we choose $k_i'\in K$ with $k_ik'_ik_i=k_i$ and put $f'=\sum_{i=1}^n k_i'1_{U_i}$,  then $f'\in C_c(X,K)$ and $ff'f=f$.
\end{proof}

\subsection{Leavitt labeled path algebras}
Let $E=(E^0,E^1,r,s)$ be a directed graph and let $\mathscr{A}$ be a set, called the alphabet. Let $\mathscr{A}^*$ denote the set of all finite words over $\mathscr{A}$ together with the empty word  $\varepsilon$, and let $\mathscr{A}^\infty$ denote all infinite words over $\mathscr{A}$.
An edge-labeling on $E$ is a surjective map $\mathscr{L}\colon E^1\to \mathscr{A}$. The pair $(E,\mathscr{L})$ is called a \emph{labeled graph}.	If $\mu = e_1e_2\cdots e_n$ is a finite or an infinite path in $E$, then $\mathscr{L}$ extends to a \emph{labeled path}  $\mathscr{L}(\mu) = \mathscr{L}(e_1)\mathscr{L}(e_2)\cdots\mathscr{L}(e_n)$. The length of a labeled path $\alpha=\mathscr{L}(\mu)$ is defined as the length of the path $\mu$ in $E$.

Let $\mathscr{L}^{n}=\mathscr{L}(E^n)$ denote the set of labeled paths of length $n$, let $\mathscr{L}^\infty$ denote the set of labeled paths of infinite length, and let $|\varepsilon|=0$. Set $\mathscr{L}^{\geq1}=\cup_{n\geq 1}\mathscr{L}^{n}$, $\mathscr{L}^*=\{\varepsilon\}\cup\mathscr{L}^{\geq 1}$, and $\mathscr{L}^{\leq \infty}=\mathscr{L}^*\cup\mathscr{L}^\infty$.

A labeled path $\alpha$ is a \emph{beginning} of a labeled path $\beta$ if $\beta=\alpha\beta'$ for some labeled path $\beta'$. If $1\leq i\leq j\leq |\alpha|$, let $\alpha_{i,j}=\alpha_i\alpha_{i+1}\ldots\alpha_{j}$ if $j<\infty$ and $\alpha_{i,j}=\alpha_i\alpha_{i+1}\ldots$ if $j=\infty$. If $j<i$ set $\alpha_{i,j}=\varepsilon$. Define $\overline{\mathscr{L}^\infty}$ to be the set of all infinite words such that all beginnings are finite labeled paths. We write $\overline{\mathscr{L}^{\leq \infty}}=\mathscr{L}^*\cup\overline{\mathscr{L}^\infty}$.

Let $\mathscr{P}(E^0)$ denote the power set of $E^0$. The \emph{relative range} of $\alpha\in\mathscr{L}^*$ with respect to  $A\in\mathscr{P}(E^0)$ is the set
\[r(A,\alpha)=
\begin{cases}
	\{r(\mu)\ |\ \mu\in E^{\ast},\ \mathscr{L}(\mu)=\alpha,\ s(\mu)\in A\}, &
	\text{if }\alpha\in\mathscr{L}^{\geq1}\\
	A, & \text{if }\alpha=\varepsilon.
\end{cases}
\]
The range $r(\alpha)$ of $\alpha$ is the set $r(\alpha)=r(E^0,\alpha)$.
We define \[\mathscr{L}(A E^1)=\{\mathscr{L}(e) \mid e\in E^1\ \mbox{and}\ s(e)\in A\}=\{a\in \mathscr{A} \mid r(A,a)\neq\emptyset\}.\]

A \emph{labeled space} is a triple $\mathscr{E} = (E,\mathscr{L},\mathscr{B})$ where $(E,\mathscr{L})$ is a labeled graph and $\mathscr{B}$ is a subset of $\mathscr{P}(E^0)$, which is closed under finite intersections and finite unions, contains $r(\alpha)$ for every $\alpha\in\mathscr{L}^*$, and is closed under relative ranges, that is, $r(A,\alpha)\in\mathscr{B}$ for all $A\in\mathscr{B}$ and all $\alpha\in\mathscr{L}^*$. We call $\mathscr{B}$ an \textit{accommodating family} for $\mathscr{E}$. A labeled space $\mathscr{E}$ is \emph{weakly left-resolving} if for all $A,B\in\mathscr{B}$ and all $\alpha\in\mathscr{L}^{\geq 1}$ we have $r(A\cap B,\alpha)=r(A,\alpha)\cap r(B,\alpha)$.  A weakly left-resolving labeled space is \emph{normal} if $\mathscr{B}$ is closed under relative complements.
A non-empty set $A\in\mathscr{B}$ is \emph{regular} if for all $\emptyset\neq B\subseteq A$, we have that $0<|\mathscr{L}(B E^1)|<\infty$. The subset of all regular elements of $\mathscr{B}$ together with the empty set is denoted by $\mathscr{B}_{reg}$.
For $\alpha\in\mathscr{L}^*$, define \[\mathscr{B}_\alpha=\mathscr{B}\cap\mathscr{P}(r(\alpha))=\{A\in\mathscr{B}\mid A\subseteq r(\alpha)\}.\] If a labeled space is normal, then  $\mathscr{B}_\alpha$ is a Boolean algebra for each $\alpha\in\mathscr{L}^*$.

\begin{Def}[Definition~3.1 of~\cite{BCGW23}]
	Let $K$ be a field and let  $\mathscr{E} = (E,\mathscr{L},\mathscr{B})$ be a normal labeled space.  The \emph{Leavitt labelled path algebra $L_K(\mathscr{E})$ associated with $\mathscr{E}$ and with coefficients in $K$} is the universal $K$-algebra with generators $\{p_A \mid A\in \mathscr{B}\}\cup\{s_a,s_a^* \mid a\in\mathscr{A}\}$ subject to the relations
	\begin{enumerate}
		\item[(i)] $p_{A\cap B}=p_Ap_B$, $p_{A\cup B}=p_A+p_B-p_{A\cap B}$ and $p_{\emptyset}=0$, for every $A,B\in\mathscr{A}$;
		\item[(ii)] $p_As_a=s_ap_{r(A,a)}$ and $s_a^{*}p_A=p_{r(A,a)}s_a^{*}$, for every $A\in\mathscr{B}$ and $a\in\mathscr{A}$;
		\item[(iii)] $s_a^*s_a=p_{r(a)}$ and $s_b^*s_a=0$ if $b\neq a$, for every $a,b\in\mathscr{A}$;
		\item[(iv)] $s_as_a^*s_a=s_a$ and $s_a^{*}s_as_a^{*}=s_a^{*}$ for every $a\in\mathscr{A}$;
		\item[(v)] For every $A\in\mathscr{B}_{reg}$,
		\[p_A=\sum_{a\in\mathscr{L}(A E^1)}s_ap_{r(A,a)}s_a^*.\]
	\end{enumerate}
\end{Def}

Leavitt labeled path algebras generalize Leavitt path algebras of graphs,  Leavitt path algebras of ultragraphs and algebraic Exel-Laca algebras.

Next, we describe a partial action associated with a labeled space and its transformation groupoid.  The Steinberg algebra of this transformation groupoid is isomorphic to the Leavitt labeled path algebra.

\subsubsection{Inverse semigroup of a labeled space}
Let  $\mathscr{E} = (E,\mathscr{L},\mathscr{B})$ be a normal labeled space and let \[S_\mathscr{E}=\{(\alpha,A,\beta)\mid \alpha,\beta\in\mathscr{L}^*\ \mbox{and}\ A\in\mathscr{B}_{\alpha}\cap\mathscr{B}_{\beta}\ \mbox{with}\ A\neq\emptyset\}\cup\{0\}.\]
Define a binary operation on $S_\mathscr{E}$ by letting $s\cdot 0= 0\cdot s=0$ for all $s\in S_\mathscr{E}$ and, if $s=(\alpha,A,\beta)$ and $t=(\gamma,B,\delta)$ are in $S_\mathscr{E}$, then \[s\cdot t=\left\{\begin{array}{ll}
	(\alpha\gamma ',r(A,\gamma ')\cap B,\delta), & \mbox{if}\ \  \gamma=\beta\gamma '\ \mbox{and}\ r(A,\gamma ')\cap B\neq\emptyset,\\
	(\alpha,A\cap r(B,\beta '),\delta\beta '), & \mbox{if}\ \  \beta=\gamma\beta '\ \mbox{and}\ A\cap r(B,\beta ')\neq\emptyset,\\
	0, & \mbox{otherwise}.
\end{array}\right. \]
Also, define $(\alpha,A,\beta)^*=(\beta,A,\alpha)$.  With these operations $S_\mathscr{E}$ is an inverse semigroup with zero element $0$; see~\cite[Proposition 3.4]{BoavaDeCastroMortari1}. The semilattice of idempotents of $S_\mathscr{E}$ is \[E(S_\mathscr{E})=\{(\alpha, A, \alpha) \mid \alpha\in\mathscr{L}^* \ \mbox{and} \ A\in\mathscr{B}_{\alpha}\}\cup\{0\}.\]

The natural order on the semilattice $E(S_\mathscr{E})$ is given by $p\leq q$ if only if $pq=p$. In our case, the order can described by noticing that if $p=(\alpha, A, \alpha)$ and $q=(\beta, B, \beta)$, then $p\leq q$ if and only if $\alpha=\beta\alpha'$ and $A\subseteq r(B,\alpha')$~\cite[ Proposition 4.1]{BoavaDeCastroMortari1}.

\subsubsection{Filters in $E(S_\mathscr{E})$}
The space on which we will define a partial action is the tight spectrum
$\mathsf{T}$ of $E(S_\mathscr{E})$ (see~\cite[Section 12]{ Exel}). One may view the tight spectrum as the closure of the set of ultrafilters (maximal proper filters) in $E(S_\mathscr{E})$, by using the correspondence between characters on $E(S_\mathscr{E})$ (with the topology of pointwise convergence) and filters in $E(S_\mathscr{E})$. Namely, characters are precisely the characteristic functions of filters.

Filters in $E(S_\mathscr{E})$ are related to labeled paths in $\mathscr{E}$ in the following way. Let $\alpha\in\overline{\mathscr{L}^{\leq \infty}}$  and $\{\mathcal{F}_n\}_{0\leq n\leq|\alpha|}$ (understanding that ${0\leq n\leq|\alpha|}$ means $0\leq n<\infty$ when $\alpha\in\overline{\mathscr{L}^{\infty}}$) be a family such that $\mathcal{F}_n$ is a filter in $\mathscr{B}{\alpha_{1,n}}$ for every $n>0$, and $\mathcal{F}_0$ is either a filter in $\mathscr{B}$ or $\mathcal{F}_0=\emptyset$. The family $\{\mathcal{F}_n\}_{0\leq n\leq|\alpha|}$ is  a  \emph{complete family for} $\alpha$ if
\[\mathcal{F}_n = \{A\in \mathscr{B}{\alpha_{1,n}} \mid r(A,\alpha_{n+1})\in\mathcal{F}_{n+1}\}\]
for all $n\geq0$.

\begin{Thm}[Theorem~4.13 of~\cite{BoavaDeCastroMortari1}]\label{thm.filters.in.E(S)}
	Let $\mathscr{E} = (E,\mathscr{L},\mathscr{B})$ be a normal labeled space. Then there is a bijective correspondence between filters in $E(S_\mathscr{E})$ and pairs $(\alpha, \{\mathcal{F}_n\}_{0\leq n\leq|\alpha|})$, where $\alpha\in\overline{\mathscr{L}^{\leq \infty}}$ and $\{\mathcal{F}_n\}_{0\leq n\leq|\alpha|}$ is a complete family for $\alpha$.
\end{Thm}

A filter $\xi$ in $E(S_\mathscr{E})$ with associated labeled path $\alpha\in\overline{\mathscr{L}^{\leq \infty}}$ is sometimes denoted by $\xi^\alpha$ to stress the word $\alpha$; in addition, the filters in the complete family associated with $\xi^\alpha$ will be denoted by $\xi^{\alpha}_n$ (or simply $\xi_n$). Specifically,
\begin{align}\label{eq.defines.xi_n}
	\xi^{\alpha}_n=\{A\in\mathscr{B} \mid (\alpha_{1,n},A,\alpha_{1,n}) \in \xi^\alpha\}.
\end{align}

\begin{Rmk}
	It follows from~\cite[Propositions 4.4 and 4.8]{BoavaDeCastroMortari1} that for a filter $\xi^\alpha$ in $E(S_\mathscr{E})$ and an element $(\beta,A,\beta)\in E(S_\mathscr{E})$ we have that $(\beta,A,\beta)\in\xi^{\alpha}$ if and only if $\beta$ is a beginning of $\alpha$ and $A\in\xi_{|\beta|}$.
\end{Rmk}

See~\cite[Theorems 5.10 and 6.7]{BoavaDeCastroMortari1}  for a characterization of tight filters in terms of labeled paths.

\subsubsection{Partial action on the tight filters}
We give a brief overview of the partial action on $\mathsf{T}$ introduced in~\cite{CastrovWyk20}.

Fix a normal labeled space $\mathscr{E} = (E,\mathscr{L},\mathscr{B})$. We describe the topology on $\mathsf{T}$, the tight spectrum of $\mathscr{E}$, inherited from the product topology on $\{0,1\}^{E(S_\mathscr{E})}$, where $\{0,1\}$ has the discrete topology. For each $e\in E(S_\mathscr{E})$, define
\begin{equation*} \label{dfn:tight.spec.basic.open.neigh}
	V_e=\{\xi\in\mathsf{T}\mid e\in\xi\},
\end{equation*}
and for $\{e_1,\ldots,e_n\}$ a finite (possibly empty) set in $E(S)$, we define
\[ V_{e\colon e_1,\ldots,e_n}= V_e\cap V_{e_1}^c\cap\cdots\cap V_{e_n}^c=\{\xi\in\mathsf{T}\mid e\in\xi,e_1\notin\xi,\ldots,e_n\notin\xi\}.\]
Sets of the form $V_{e\colon e_1,\ldots,e_n}$ form a basis of compact open sets for a Hausdorff topology on $\mathsf{T}$~\cite[Remark 4.2]{BCM20}.

Let $\mathbb{F}$ be the free group generated by $\mathscr{A}$ (identifying the identity of $\mathbb{F}$ with $\varepsilon$). For every $t\in \mathbb{F}$ there is a clopen set $V_t\subseteq \mathsf{T}$, which is compact if $t\neq\varepsilon,$ and a homeomorphism $\varphi_t\colon V_{t^{-1}}\to V_t$ such that
\begin{equation*} \label{eq:PartialAction}
	\varphi=(\{V_t\}_{t\in\mathbb{F}},\{\varphi_t\}_{t\in\mathbb{F}})
\end{equation*}
is a topological partial action of $\mathbb{F}$ on $\mathsf{T}$~\cite[Proposition 3.12]{CastrovWyk20}. In particular, $V_\varepsilon=\mathsf{T}$ and if
$\alpha,\beta\in \mathscr{L}^*$, then $V_\alpha=V_{(\alpha,r(\alpha),\alpha)}, V_{\alpha^{-1}}=V_{(\varepsilon,r(\alpha),\varepsilon)}$, and $V_{(\alpha\beta^{-1})^{-1}}= \varphi_{\beta^{-1}}^{-1}(V_{\alpha^{-1}})$, with $V_{(\alpha\beta^{-1})^{-1}}\neq\emptyset$ if and only if $r(\alpha)\cap r(\beta)\neq\emptyset$~\cite[Lemma 3.10]{CastrovWyk20}. If $V_t\neq \emptyset$, then $t=\alpha\beta^{-1}$ for some $\alpha,\beta\in \mathscr{L}^*$~\cite[Lemma 3.11(ii)]{CastrovWyk20}, in which case
\begin{equation*}
	V_{\alpha\beta^{-1}}=V_{(\alpha,r(\alpha)\cap r(\beta),\alpha)}
\end{equation*}
by~\cite[Lemma~4.3]{BCGW21}. If $\xi\in V_{\beta\alpha^{-1}}$ with associated word $\beta\gamma$, then the associated word of $\varphi_{\alpha\beta^{-1}}(\xi)$ is $\alpha\gamma$ and, for $0\leq n\leq|\gamma|$, we have that
\begin{equation*}
	\varphi_{\alpha\beta^{-1}}(\xi)_{|\alpha|+n}=\{A\cap r(\alpha\gamma_{1,n})\mid A\in\uparrow_{\mathscr{B}_{\gamma_{1,n}}}\xi_{|\beta|+n}\},
\end{equation*}
where $\uparrow_{\mathscr{B}_{\gamma_{1,n}}}\xi_{|\beta|+n}$ is the upper set of $\mathscr{B}_{\gamma_{1,n}}$ generated by $\xi_{|\beta|+n}$ 
\begin{equation*}
	\varphi_{\beta^{-1}}(\xi)_{n}=\uparrow_{\mathscr{B}_{\gamma_{1,n}}}\xi_{|\beta|+n}.
\end{equation*}

The transformation groupoid associated with the partial action  $\Phi=(\{V_t\}_{t\in\mathbb{F}},\{\phi_t\}_{t\in\mathbb{F}})$ is
\begin{equation}\label{eqn:Groupoid2}
	\mathscr{G}=\{(\xi,t,\eta)\in\mathsf{T}\times\mathbb{F}\times\mathsf{T} \mid \eta\in V_{t^{-1}}, \text{ and } \xi=\phi_t(\eta)\}
\end{equation}
with products and inverses defined by
\[ (\xi,s,\eta)(\eta,t,\rho)=(\xi,st,\rho) \text{ and }  (\xi,t,\eta)^{-1}=(\eta,t^{-1},\xi),\]
respectively. Endow $\mathscr{G}$ with the relative topology from the product topology on $\mathsf{T}\times\mathbb{F}\times\mathsf{T}$. The unit space $\mathscr{G}^0$ is identified with $\mathsf{T}$. By~\cite[Theorem 6.1]{BCGW23}, the Leavitt labeled graph algebra of $L_K(\mathscr{E})$ is isomorphic to $K\mathscr{G}$. The transformation groupoid $\mathscr{G}$ is isomorphic to a Deaconu-Renault groupoid associated with $\mathscr{E}$~\cite[Theorem 5.5]{CastrovWyk20}. Thus, all isotropy subgroups of $\mathscr{G}$ are subgroups of the integers.

The Leavitt path algebra of an acyclic graph is a direct limit of Leavitt path algebras of finite acyclic graphs~\cite[Proposition 1.6.15]{LeavittBook}, which implies the Leavitt path algebra is locally matricial~\cite[Proposition 2.6.20]{LeavittBook}. This fact is used to  prove regularity of Leavitt path algebras of acyclic graphs. The case of labeled spaces is more intricate. For example, a row-finite graph may produce two labeled spaces, one with a regular Leavitt labeled path algebra and the other with a non-regular Leavitt labeled path algebra, depending on the labeling and accommodating family chosen (see Example~\ref{e:diff_algs}). In addition, it is not known if, or under which conditions, Leavitt labeled path algebras are locally matricial. 

Using our groupoid characterization we obtain some regularity results, but not an equivalence. For graphs, being acyclic is equivalent to the boundary path groupoid being principal. In the transformation groupoid of a labeled space, a tight filter $\xi^\alpha$ has non-trivial isotropy in $\mathscr{G}$ if and only if there exist $\beta,\gamma\in\mathscr{L}^{\geq 1}$ such that $\alpha=\beta\gamma^{\infty}$ and for all $n\in\mathbb{N}\cup\{0\}$ and all $A\in\xi_{|\beta\gamma^n|}$ we have that $A\cap r(A,\gamma)\neq\emptyset$~\cite[Proposition 6.10]{CastrovWyk20}. The negation of this characterization then characterizes when the groupoid is principal. However, assuming the existence of a tight filter with associated word $\alpha=\beta\gamma^{\infty}$ such that there exists $n\in\mathbb{N}\cup\{0\}$ and  $A\in\xi_{|\beta\gamma^n|}$ with $A\cap r(A,\gamma)=\emptyset$ seems to be an obstacle to finding an equivalence between principality and regularity.
Nonetheless, we have the following results.

\begin{Thm}\label{t:labeled_reg_princ}
	Let $\mathscr{E}=(E,\mathscr{L},\mathscr{B})$ be normal labeled space and $K$ a field. If $L_K(\mathscr{E})$ is regular, then $\mathscr{G}_\mathscr{E}$ is principal.
\end{Thm}
\begin{proof}
	Suppose $L_K(\mathscr{E}) \cong K\mathscr{G}$ is regular. Then the isotropy subgroups of $\mathscr{G}$ are locally finite by Theorem~A. Since all isotropy subgroups of $\mathscr{G}$ are infinite cyclic, it follows that $\mathscr{G}$ is principal.
\end{proof}

\begin{Cor}\label{c:labeled_reg_princ}
	If $L_K(\mathscr{E})$ is regular, then for all $\xi^\alpha\in \mathsf{T}$, either $\alpha \neq \beta\gamma^{\infty}$ for any $\beta,\gamma\in \mathscr{L}^*$, or if $\alpha = \beta\gamma^{\infty}$ then there is $n\in \mathbb{N}$ and $A\in\xi_{|\beta\gamma^n|}$ such that $A\cap r(A,\gamma)=\emptyset$.
\end{Cor}
\begin{proof}
	The result follows from Theorem~\ref{t:labeled_reg_princ} and the characterization of nontrivial isotropy in $\mathscr{G}_{\mathscr{E}}$~\cite[Proposition 6.10]{CastrovWyk20}.
\end{proof}

\begin{Thm}\label{t:labeled_reg}
	Let $\mathscr{E}=(E,\mathscr{L},\mathscr{B})$ be normal labeled space and $K$ a field. Suppose that if $\xi^\alpha\in \mathsf{T}$ and $\alpha \neq \beta\gamma^{\infty}$ for any $\beta,\gamma\in \mathscr{L}^{\geq 1}$. Then $L_K(\mathscr{E})$ is regular.
\end{Thm}
\begin{proof}
	Let $F$ be a finite subgraph of $E$. Let $\mathscr{L}_F\colon F^1\to \mathscr{A}_F$ denote the restriction of $\mathscr{L}$ to $F^1\subset E^1$, where $\mathscr{A}_F = \mathscr{L}(F^1)\subset \mathscr{A}$. Then $(F,\mathscr{L}_F)$ is a labeled subgraph. Let $\mathscr{B}(F) = \mathscr{B}\cap \mathscr{P}(F^0)$. Then $\mathscr{F}=(F,\mathscr{L}_F,\mathscr{B}(F))$ is a normal labeled space, which is set-finite in the sense that $|\mathscr{L}_F(A F^1)|< \infty$ for all $A\in \mathscr{B}(F)$. Notice that since we assume  $\alpha \neq \beta\gamma^{\infty}$ for any $\beta,\gamma\in \mathscr{L}^*$, it follows that $|\mathscr{L}^*_F|<\infty$. And, since $F^0$ is finite, it follows that $|\mathscr{B}(F)|<\infty$. Hence, $|E(S_\mathscr{F})|<\infty$.
	Let
	\begin{align*}
		\mathscr{G}(F) = \bigcup \{ &  \left( V_{e\colon e_1,\ldots,e_m}  \times\{\alpha\beta^{-1}\}\times
		V_{f\colon f_1,\ldots,f_n} \right)   \cap \mathscr{G}_\mathscr{E} \mid  \\
		& e,e_i,f,f_j\in E(S_\mathscr{F}), \text{ and } \alpha,\beta\in \mathscr{L}^*_F \}.
	\end{align*}

	We claim that $\mathscr{G}(F)$ is a compact open subgroupoid of $\mathscr{G}_\mathscr{E}$. Notice that $E(S_\mathscr{F})\subseteq E(S_\mathscr{E})$, $\mathscr{B}_F\subseteq \mathscr{B}$, and $\{\alpha\beta^{-1}\mid \alpha,\beta\in \mathscr{L}^*_F\} \subseteq \mathbb{F}$ (where $\mathbb{F}$ is the free group generated by $\mathscr{A}$). Thus, since each $V_{e\colon e_1,\ldots,e_m}$ is  compact open in $\mathsf{T}$, it follows that $\mathscr{G}(F)$ is an open subset of $\mathscr{G}_\mathscr{E}$.  It is clear that $\mathscr{G}(F)$ is closed under inverses.
	
	We show that $\mathscr{G}(F)$ is closed under taking products. Let $(\xi,s,\eta),(\eta,t,\rho)\in \mathscr{G}(F)$ with the product $(\xi,s,\eta)(\eta,t,\rho)$ defined in $\mathscr{G}_\mathscr{E}$. Then there are $p,q,m,n\in \mathbb{N}\cup\{0\}$ and $d,d_1,\ldots,d_p, e,e_1,\ldots,e_m, f,f_1,\ldots,f_n, g,g_1,\ldots,g_q\in E(S_\mathscr{F})$ such that
	\[(\xi,s,\eta)\in V_{d\colon d_1,\ldots,d_p}\times \{s\}\times \times V_{e\colon e_1,\ldots,e_m}, \text{ and } \]
	\[(\eta,t,\rho) \in V_{f\colon f_1,\ldots,f_n}\times \{t\}\times \times V_{g\colon g_1,\ldots,g_q}.\]
	Suppose $e=(\alpha, A,\alpha)$ and $f=(\beta,B,\beta)$ with $\alpha,\beta\in \mathscr{L}_F^*\subset \mathscr{L}^*$ and $A,B\in\mathscr{B}_F$. Hence, $\eta=\eta^{\alpha\gamma} =\eta^{\beta\delta}$, for some $\alpha\gamma,\beta\delta \in\overline{\mathscr{L}^{\leq \infty}}$ with $A\in \eta_{|\alpha|}$ and $B\in \eta_{|\beta|}$. The one-to-one correspondence of Theorem~\ref{thm.filters.in.E(S)} between tight filters in $\mathsf{T}$ and pairs consisting of a labeled path and a complete family of filters implies that $\alpha\gamma=\beta\delta$. Hence, $\alpha=\beta\alpha'$ or $\beta=\alpha\beta'$.
	
	Suppose first that $\alpha=\beta\alpha'$. Then $\beta\delta = \beta\alpha'\gamma$. By~\cite[Lemma 5.3]{CastrovWyk20}, there are $\tau,\nu\in  \mathscr{L}_F^*$ such that
	\[(\xi,s,\eta) =  (\xi^{\tau\gamma},\tau\alpha^{-1},\eta^{\alpha\gamma})=(\xi^{\tau\gamma},\tau(\beta\alpha')^{-1},\eta^{\beta\alpha'\gamma})\]
	and
	\[(\eta,t,\rho) = (\eta^{\beta\delta}, \beta\nu^{-1}, \rho^{\nu\delta}) = (\eta^{\beta\alpha'\gamma}, \beta\alpha'(\nu\alpha')^{-1}, \rho^{\nu\alpha'\gamma}).\]
	Thus,
	\begin{eqnarray}\label{e:labeled_open_subgrpd}
		(\xi,s,\eta)(\eta,t,\rho)  &=& (\xi^{\tau\gamma},\tau(\nu\alpha')^{-1}, \rho^{\nu\alpha'\gamma}) \nonumber\\
									&\in&  V_{d\colon d_1,\ldots,d_p}\times\{\tau(\nu\alpha')^{-1}\} \times V_{g\colon g_1,\ldots,g_q}.
	\end{eqnarray}
	
	On the other hand, if $\beta=\alpha\beta'$, then $\alpha\gamma = \alpha\beta'\delta$. Similarly to above, we see that
	\[(\xi,s,\eta)(\eta,t,\rho)  = (\xi^{\tau\beta'\delta},\tau\beta' \nu^{-1},\rho^{\nu\delta})\in  V_{d\colon d_1,\ldots,d_p}\times\{\tau\beta' \nu^{-1}\} \times V_{g\colon g_1,\ldots,g_q}. \]
	Hence $\mathscr{G}(F)$ is closed under multiplication from $\mathscr{G}_\mathscr{E}$, and thus an open subgroupoid of $\mathscr{G}_\mathscr{E}$.
	
	Since $|\mathscr{L}^*_F|<\infty$, $|E(S_\mathscr{F})|<\infty$ and $V_{e\colon e_1,\ldots,e_m}$ is compact for all $e,e_i\in E(S_\mathscr{F})$, it follows that $\mathscr{G}(F)$ is a finite union of compact sets, and therefore itself compact. Then, $\mathscr{G}_{\mathscr{E}}$ is a directed union of $\mathscr{G}(F)$, ranging over all finite subgraphs $F$ of $E$, and we conclude that $\mathscr{G}_{\mathscr{E}}$ is approximately quasi-compact.

	Notice that the assumption $\alpha \neq \beta\gamma^{\infty}$ for any $\beta,\gamma\in \mathscr{L}^{\geq 1}$ implies that all isotropy in $\mathscr{G}_{\mathscr{E}}$ are trivial~\cite[Proposition 6.10]{CastrovWyk20}. Thus, the order of no isotropy element is divisible by the characteristic of $K$. Hence, $L_K(\mathscr{E})\cong K\mathscr{G}_{\mathscr{E}}$ is regular by Theorem~A.
\end{proof}

\begin{Example}\label{e:diff_algs}
	We give an example that illustrates the significant role that the labeling and the accommodating family plays in the algebraic structure of algebras associated with labeled spaces. Let $E$ be the graph with vertices $E^0= \{v_i\mid i\in \mathbb{Z}\}$ and edges $E^1=\{e_i\mid  i\in \mathbb{Z}\}$ such that $r(e_i)=s(e_{i+1})$ for all $i\in \mathbb{Z}$. $E$ is an acyclic row-finite graph. Let $\mathscr{A} = \{a\}$ and let $\mathscr{L}(e_i)=a$ for all $i\in\mathbb{Z}$. Let $\mathscr{B}$ be the Boolean algebra generated by all finite and cofinite subsets of $E^0$. Then $\mathscr{B}$ is an accommodating family for the labeled graph $(E,\mathscr{L})$ and $\mathscr{E} = (E,\mathscr{L},\mathscr{B})$ is a normal labeled space. Here, $\mathscr{L}^*=\{a^n\mid  n\in \mathbb{N}\}\cup \{\epsilon\}$ and, therefore, $a^\infty\in \overline{\mathscr{L}^{\leq \infty}}$. This implies that $\mathscr{G}_{\mathscr{E}}$ is not principal and, therefore,  $L_K(\mathscr{E})$ is not regular by Theorem~\ref{t:labeled_reg_princ}. The underlying graph, however, is acyclic and thus its Leavitt path algebra $L_K(E)$ is regular by Theorem~\ref{Leavitt_reg}.
	
	Now let $\mathscr{A}' = E^1$, let $\mathscr{C}$ be the set of all finite subsets of $E^0$ and let $\mathscr{K}\colon E^1\to\mathscr{A}'$ be the identity map. Then, $\mathscr{E}'=(E,\mathscr{K}, \mathscr{C})$ is a normal label space. In this instance, $L_K(\mathscr{E}') \cong L_K(E)$ (by~\cite[Example 7.1]{BCGW23}) and  is, therefore, regular.
\end{Example}

\begin{Rmk}
	Theorem~\ref{t:labeled_reg} gives a partial converse to Corollary~\ref{c:labeled_reg_princ} and, thus, to Theorem~\ref{t:labeled_reg_princ}. However, we don't see how to show that the property ``if $\alpha\neq \beta\gamma^{\infty}$, then there exists $n\in\mathbb{N}\cup\{0\}$ and  $A\in\xi_{|\beta\gamma^n|}$ such that $A\cap r(A,\gamma)=\emptyset$'' implies regularity. In fact, we do not even have an example of a labeled space that exhibits this property, nor do we know whether it influences regularity of $L_K(\mathscr{E})$.
\end{Rmk}

Every Leavitt labeled path algebra $L_R(\mathscr{E})$ is $\mathbb{Z}$-graded, with homogeneous components given by
\[L_R(\mathscr{E})_n = \mathrm{span}_K\{s_\alpha p_A s_\beta^* \mid \alpha,\beta\in\mathscr{L}^*,\ A\in\mathscr{B}_{\alpha}\cap\mathscr{B}_{\beta} \ \mbox{and} \ |\alpha|-|\beta|=n\}\]
for $n\in \mathbb{Z}$~\cite[Proposition 3.8]{BCGW23}.

\begin{Thm}\label{labeled_gr_regular}
	Let $\mathscr{E}=(E,\mathscr{L},\mathscr{B})$ be normal labeled space and $K$ a field. Then the Leavitt labeled path algebra $L_K(\mathscr{E})$ is graded regular.
\end{Thm}
\begin{proof}
To ease notation, let $\mathscr{G}=\mathscr{G}_\mathscr{E}$ denote the transformation groupoid of the partial action associated with $\mathscr{E}$. The map $c\colon (\xi^{\alpha\gamma},\alpha\beta^{-1},\eta^{\beta\gamma})\mapsto |\alpha|-|\beta|$ is a continuous cocycle $c\colon \mathscr{G}\to\mathbb{Z}$ that induces the $\mathbb{Z}$-grading  on $L_R(\mathscr{E})$ via the isomorphism $L_K(\mathscr{E})\cong K\mathscr{G}$. Then
\[\mathscr{G}_0 = \{(\eta^{\alpha\gamma}, \alpha\beta^{-1},\eta^{\beta\gamma})\mid \alpha,\beta\in\mathscr{L}^*,\,r(\alpha)\cap r(\beta)\neq \emptyset, |\alpha|=|\beta|\}.\]

If $(\eta, t, \eta)\in \mathscr{G}_0$ is an isotropy group element, then $\eta = \eta^{\alpha\gamma^\infty}$, for some $\alpha\in\mathscr{L}^*$ and $\gamma\in\mathscr{L}^{\geq 1}$, and either $t=\alpha\gamma\alpha^{-1}$ or $t=\alpha\gamma^{-1}\alpha^{-1}$~\cite[Remark 8.2]{BCGW23}. Without loss of generality, suppose $t=\alpha\gamma\alpha^{-1}$. Then, $c(\eta, t, \eta) = |\gamma|=0$, which implies that $t=\varepsilon$. Hence $\mathscr{G}_0$ is has trivial isotropy groups.

Let $F$ be a finite subgraph of $E$, let $\mathscr{L}_F\colon F^1\to \mathscr{A}$ be the restriction of $\mathscr{L}$ to $F^1\subset E^1$, and let $\mathscr{B}(F) = \mathscr{B}\cap \mathscr{P}(F^0)$. Then $\mathscr{F}=(F,\mathscr{L}_F,\mathscr{B}(F))$ is a normal labeled space.
For $n\geq 0$, let
\[\mathscr{G}_0(F,n) = \bigcup_{\substack{
		\alpha,\beta\in\mathscr{L}_F^* \\
		A\in\mathscr{B}(F)\\
		A\cap r(\alpha)\cap r(\beta)\neq\emptyset \\
		|\alpha|=|\beta|\leq n}}
		V_{(\alpha,A,\alpha)}\times\{\alpha\beta^{-1}\}\times V_{(\beta,A,\beta)} \]

Then $\mathscr{G}_0(F,n)$ is an open subgroupoid of $\mathscr{G}_0$ by (\ref{e:labeled_open_subgrpd}). Moreover, given $n$ and that $F$ is finite, it follows that there are only finitely many choices of $A$, $\alpha$, $\beta$ and $B$. 
Thus, $\mathscr{G}_0(F,n)$ is compact, as it is a finite union of compact sets.

If $F\subseteq F'$ and $n\leq n'$, then 
$\mathscr G_0(F,n)\subseteq \mathscr G_0(F',n')$.  Thus,  $\mathscr G_0$ is approximately quasi-compact. By Theorem~\ref{t:graded.vnr}, $L_K(\mathscr{E})\cong K\mathscr{G}_\mathscr{E}$ is graded regular.
\end{proof}

\subsection{Deaconu-Renault groupoids}
A monoid $P$ is called \emph{left reversible} (or right Ore) if $aP\cap bP\neq\emptyset$ for all $a,b\in P$.  If $P$ is cancellative and left reversible, then it embeds in its group of fractions $G$.  One can construct $G$ as $(P\times P)/{\sim}$ where $(p,q)\sim (p',q')$ if the exist $x,y\in P$ with $(px,qx)=(p'y,q'y)$.  One writes $pq\inv$ for the class of $(p,q)$ and the product is given by $pq\inv rs\inv = px(sy)\inv$ where $x,y\in P$ with $qx=ry$.  The embedding of $P$ into $G$ is $p\mapsto pe\inv$, which we write as $p$ from now on.  Note that $G=PP\inv$.

Following~\cite{RenaultWilliams}, we consider right actions of $P$ on a space $X$ by partial local homeomorphisms.  So we have a partially defined mapping $X\times P\to X$ (written $(x,p)\mapsto xp$) such that if $U_p$ is the set of $x\in X$ with $xp$ defined and $V_p=\{xp\mid x\in U_p\}$, then $U_p$ and $V_p$ are open, and the following hold:
\begin{enumerate}
  \item  $x\mapsto xp$ is a local homeomorphism $U_p\to V_p$ for all $p\in P$.
  \item  $U_e=X$ and $xe=x$ for all $x\in X$.
  \item  $x\in U_{pq}$ if and only if $x\in U_p$ and $xp\in U_q$, in which case $(xp)q=x(pq)$.
\end{enumerate}
The action is said to be \emph{directed} if, for all $p,q\in P$ with $U_p\cap U_q\neq\emptyset$, there is $r\in pP\cap qP$ with $U_r=U_p\cap U_q$.  Notice that left reversibility is equivalent to this condition if $U_p=X$ for all $p\in P$.  More generally, if $F\subseteq P$, then $F$ is said to be \emph{action directed} if given $p,q\in F$ with $U_p\cap U_q\neq \emptyset$, there is $r\in F\cap pP\cap qP$ with $U_p\cap U_q=U_r$.  It is shown in~\cite[Claim~3]{RenaultWilliams} that every finite subset of $P$ is contained in an action-directed finite subset of $P$.

From now on we assume that $X$ is locally compact, Hausdorff and totally disconnected.

For example, if $E$ is a directed graph, then there is an action of $\mathbb N$ on $\partial E$ by local homeomorphisms by putting $xn=\sigma^n(x)$ where $\sigma$ is the shift map.  Here $U_n =\partial E^{\geq n}$.  The action-directed condition is trivial since $r$ can be taken to be the maximum of $p,q$.

The \emph{Deaconu-Renault groupoid} associated to the action of $P$ on $X$ is
\[\mathscr G(X,P) = \{(x,g,y)\in X\times G\times Y\mid \exists p,q\in P, xp=yq, g=pq\inv\}.\]
A pair $(x,g,y),(x',g',y')$ is composable if and only if $y=x'$, in which case $(x,g,y)(x',g',y')=(x,gg',y')$.  The inversion is given by $(x,g,y)\inv = (y,g\inv ,x)$.  The unit space is $\mathscr G(X,P)\skel 0=\{(x,e,x)\mid x\in X\}$, which we can identify with $X$.

A basis for the topology on $\mathscr G(X,P)$ consists of all sets of the form $D(U,p,q,V)=\{(x,pq\inv, y)\mid x\in U,y\in V,xp=yq\}$ with $U$ and $V$ open subsets of $X$.  Such a set is a compact open bisection if $U\subseteq U_p$, $V\subseteq V_q$ are compact open, $p$ is injective on $U$, $q$ is injective on $V$ and $Up=Vq$.  These compact open bisections form a basis for the topology, and so the groupoid is ample.  The topology on $\mathscr G(X,P)$ is finer than the product topology on $X\times G\times X$ but coincides with the product topology on sets of the form $D(U,p,q,V)$ (see~\cite[Lemma~5.11]{RenaultWilliams}).  The groupoid $\mathscr G(X,P)$ is Hausdorff.
There is a continuous cocycle $c\colon \mathscr G(X,P)\to G$ given by $c(x,g,y)=g$ (see~\cite[Proposition~5.12]{RenaultWilliams}).

Notice that the boundary path groupoid $\mathscr G_E$ for a graph $E$ is a Deaconu-Renault groupoid for the action of $\mathbb N$ via the shift map.  Similarly, the boundary path groupoid of a higher rank graph (or $k$-graph) is a Deanconu-Renault groupoid for an action of $\mathbb N^k$ on the boundary path space~\cite{RenaultWilliams}.

Our goal is to show that $\mathscr G(X,P)$ is graded regular with respect to the $G$-grading coming from the canonical cocycle provided that each set $U_g$ is clopen.  This occurs in particular if each $U_g=X$.  It also happens in the case of row-finite graphs and row-finite higher rank graphs.  In this setting, the finite paths in the boundary path space are isolated points.  Since every infinite path is in the domain of any shift map, it follows that the domains are clopen in the row-finite case.  The following result recovers graded regularity for Leavitt path algebras of row-finite graphs and also establishes graded regularity for row-finite higher ranks graphs.

\begin{Thm}\label{t:dr.graded}
Let $R$ be a regular unital ring and let $\mathscr G(X,P)$ be the Deaconu-Renault groupoid of the directed action of a cancellative left reversible monoid $P$ on a locally compact, Hausdorff and totally disconnected space $X$ by partial local homeomorphisms.  Assume that the domain $U_p$ of the action of $p$ is clopen for all $p\in P$.  Then $R\mathscr G(X,P)$ is graded regular with respect to the $G$-grading coming from the canonical cocycle, where $G$ is the group of fractions of $P$.
\end{Thm}
\begin{proof}
Note that $c\inv(e) = \{(x,e,y)\in \mathscr G(X,P)\}$, and hence it is a principal groupoid.  So it suffices by Theorem~\ref{t:graded.vnr} to show that $c\inv(e)$ is approximately quasi-compact.
Let $W$ be a compact open subset of $X$, and let $F$ be an action-directed finite subset of $P$.  Put $\mathscr G(W,F)=\{(x,e,y)\mid x,y\in W, \exists p\in F,xp=yp\}$.  We claim that $\mathscr G(W,F)$ is a compact open subgroupoid of $c\inv(e)$.  First of all,
 \[\mathscr G(W,F) = \bigcup_{p\in F}D(W\cap U_p,p,p,W\cap U_p)\] and hence is compact open, since $W\cap U_p$ is compact open because $U_p$ is clopen, and the topology on $D(W\cap U_p,p,p,W\cap U_p)$ is the product topology on $(W\cap U_p)\times \{e\}\times (W\cap U_p)$ by~\cite[Lemma~5.11]{RenaultWilliams}.  Indeed, since $X$ is Hausdorff, the set of elements of $U_p\times U_p$ with $xp=yp$ is closed in $U_p\times U_p$ and hence a compact subset of $(W\cap U_p)\times (W\cap U_p)$ as $U_p$ is clopen.  It remains to check that $\mathscr G(W,F)$ is a subgroupoid. It is clearly closed under inversion.  But if $(x,e,y),(y,e,z)\in \mathscr G(W,F)$, then we can find $p,q\in F$ with $xp=yp$ and $yq=zq$.  Since $F$ is action directed, we can find $r\in F\cap pP\cap qP$ with $U_r=U_p\cap U_q$.  Then $y\in U_r$.  Let $r=pa=qb$.  Then $yr$ defined implies that $a$ is defined on $yp=xp$ and $b$ is defined on $yq=zq$.  Therefore, $xr=xpa$ and $zr=zqb$ are defined and $xr=xpa=ypa=yr=yqb=zqb=zr$.  Thus $(x,e,z)\in \mathscr G(W,F)$.  We conclude that $\mathscr G(W,F)$ is a compact open subgroupoid. If $F,F'$ are action-directed finite subsets of $P$ and $W,W'$ are compact open subsets of $X$, then there is an action-directed finite subset $F''$ with $F\cup F'\subseteq F''$ by~\cite[Claim~3]{RenaultWilliams}.  Then $\mathscr G(W,F),\mathscr G(W',F')\subseteq \mathscr G(W\cup W',F'')$ by construction.  If $(x,e,y)\in c\inv (e)$, then we can find a compact open set $W$ containing $x,y$ and we can find $p\in P$ with $xp=yp$.  Clearly, $\{p\}$ is action directed.  It follows that $c\inv (e)$ is the directed union of the compact open subgroupoids $\mathscr G(W,F)$ and hence is approximately quasi-compact.  The result follows.
\end{proof}

\subsection{Self-similar group actions on graphs and Nekrashevych-Exel-Pardo algebras}
Nekrashevych introduced algebras associated to self-sim\-i\-lar groups~\cite{Nekcstar,Nekrashevychgpd} and these were generalized by Exel-Pardo~\cite{ExelPardoSelf} to self-similar actions on graphs.  To keep things simple, we will work with finite directed  graphs with no sinks.

Let $E$ be a finite directed graph with no sinks. Then $\partial E=E^{\infty}$ is compact Hausdorff and totally disconnected.   If $G$ is a discrete group, then a self-similar action of $G$ on $E$ consists of an action of $G$ on $E$ by digraph automorphisms and a mapping $G\times E^1\to G$, written $(g,e)\mapsto g|_e$ and called the \emph{section} of $g$ at $e$, satisfying:
\begin{enumerate}
\item $g(v)=g|_e(v)$ for all $v\in E^0$ and $e\in E^1$;
\item $(gh)|_e = g|_{h(e)}h|_e$.
\end{enumerate}
One easily checks that $G$ acts on $E^*$ via the recursive formula $g(ex) = g(e)g|_e(x)$ for $ex\in E^*$ and with $e\in E^1$.  This product makes sense because of the first axiom and it is an action by the second.  It follows that the action of $G$ on $E^*$ induces an action of $G$ on $\partial E=E^\infty$.  The action of $g\in G$ on an infinite path can be computed recursively via the formula $g(ex) =g(e)g|_e(x)$ for $ex\in E^\infty$ and $e\in E^1$.  Note that this action preserves the length of the longest common prefix of two infinite paths and hence $G$ acts on $\partial E$ by homeomorphisms.  We say that the self-similar group action is faithful if the action of $G$ on $\partial E$ is faithful.

The classical notion of self-similar groups~\cite{selfsimilar} concerns the special case when $E$ has a single vertex and at least two edges.  The following notion is standard for self-similar group actions for single-vertex graphs and we define it now in the general case.  The mapping $G\times E^1\to G$ given by $(g,e)\mapsto g|_e$ extends to an action of the free monoid on $E^1$ on the set $G$ by mappings.  We say that $g\in G$ is \emph{finite state} if the forward orbit of $g$ under this action is finite.  We say that $G$ is \emph{finite state} if each element of $G$ is finite state.  It is easy to check (and well known for the case when $E$ is a single-vertex graph~\cite{selfsimilar}) that if $X\subseteq G$ is any subset closed under taking sections (i.e., under the action of the free monoid), then the subgroup $H=\langle X\rangle$ acting on $E$ is self-similar, that is, $h|_e\in H$ for all $h\in H$ and $e\in E^1$. Moreover, if $|X|<\infty$, then each element of $H$ is finite state.   Finitely generated, finite state self-similar groups are also known as automaton groups.

There is an algebra associated to any self-similar group action, generalizing the Leavitt path algebra~\cite{Nekcstar,ExelPardoSelf}.  It can be given by generators and relations but we give the groupoid description.  Exel and Pardo~\cite{ExelPardoSelf} construct an inverse semigroup $S_{G,E}$ and then use a groupoid of germs construction to build the groupoid.  The inverse semigroup $S_{G_,E}$ can be defined as the inverse semigroup with zero whose nonzero elements consists of expressions of the form $\alpha g\beta^*$ with $\alpha,\beta\in E^*$ and $r(\alpha)=gr(\beta)$. The product is given by
\[\alpha g\beta^*\cdot \lambda h\nu^*=\begin{cases} \alpha g(\gamma)g|_{\gamma}\nu ^*, & \text{if}\ \lambda=\beta\gamma,\\ \alpha (h\inv|_{\gamma})\inv (\nu h\inv(\gamma))^*, & \text{if}\ \beta=\lambda\gamma,\\ 0, & \text{else.}\end{cases}\] The involution is given by $(pgq^*)^* = qg\inv p^*$.   If $E$ has a single vertex $v$, we can identify the group of units of $S_{G,E}$ with $G$, but otherwise $S_{G,E}$ is not monoid.

There is an action of $\theta\colon S_{G,E)}\to I_{\partial E}$ defined as follows.  The zero acts as the empty map.  If $\alpha g\beta^*\in S_{G,E}$ with $gr(\beta)=r(\alpha)$, the corresponding partial homeomorphism has domain $Z(\beta)$, range $Z(\alpha)$ and $\theta_{\alpha g\beta^*}(\beta \eta) = \alpha g(\eta)$ for $\eta\in Z(r(\beta))$.

The groupoid $\mathscr G_{G,E}$ is the groupoid of germs of the action of $S_{G,E}$ on $\partial E$.  If $G$ is trivial, this recovers the boundary path groupoid $\mathscr G_E$.  The groupoid $\mathscr G_{G,E}$ is not usually Hausdorff.  The \emph{Exel-Pardo algebra} of $(G,E)$ over $K$ is $K\mathscr G_{G,E}$.  When $G$ is trivial, this is just $L_K(E)$.

There is a well-defined continuous cocycle $c\colon \mathscr G_{G,E}\to \mathbb Z$ given by putting $c([\alpha g\beta\inv ,x]) = |\alpha|-|\beta|$ and this induces a $\mathbb Z$-grading on $K\mathscr G_{G,E}$.  The following is an analogue of Theorem~\ref{Leavitt_graded_regular} for Exel-Pardo algebras.

\begin{Thm}
Let $(G,E)$ be a faithful finite state self-similar group action of a group $G$ on a finite directed graph $E$ with no sinks and let $R$ be a unital ring.  Then the Exel-Pardo algebra $K\mathscr G_{G,E}$ is graded regular with respect to its usual $\mathbb Z$-grading if and only if $G$ is locally finite.
\end{Thm}
\begin{proof}
Suppose first that $K\mathscr G_{G,E}$ is graded regular.  There is a homomorphism $\psi\colon G\to \Gamma_c((\mathscr G_{G,E})_0)$ given by $\psi(g) = \{[g(v)gv^*,\nu]\mid v\in Z(v), v\in E^0\}$.  Moreover, $\psi$ is injective because the action is faithful, and so if $g\neq 1$, then $g(x)\neq x$ for some $x\in \partial E$, whence $\ran([g(s(x))gs(x)^*,x])=gx\neq x=\sour([g(s(x))gs(x)^*,x])$.  But $\psi(1) = G\skel 0$.  So $\psi(g)\neq \psi(1)$.  By Theorem~A, $\Gamma_c((\mathscr G_{G,E})_0)$ is locally finite and hence $G\cong \psi(G)$ is locally finite.

For the converse, let $(S_{G,E})_0$ be the inverse subsemigroup of $S_{G,E}$ consisting of $0$ and all elements of the form $\alpha g\beta^*$ with $r(\alpha)=r(\beta)$ and $|\alpha|=|\beta|$.     Then $(S_{G,E})_0$ contains all the idempotents of $S_{G,E}$, from which it follows that $(\mathscr G_{G,E})_0$ is the groupoid of germs of the action of $(S_{G,E})_0$ on $\partial E$.  Therefore, by Corollary~\ref{c:gpd.germs} and Theorem~\ref{t:graded.vnr} it suffices to show that $(S_{G,E})_0$ is locally finite.  Since $G$ is finite state, if $X\subseteq G$ is finite then replacing $X$ by its forward orbit under the action of the free monoid on $E$, we see that $G$ is the direct limit of subgroups generated by finite subsets $Y$ of $G$ closed under sections. Since $G$ is locally finite, such a finite subset $Y$ generates a finite self-similar subgroup of $G$.  Thus we are reduced to showing that $(S_{G,E})_0$ is locally finite when $G$ is finite.  If $F$ is any finite subgraph of $E$, then since $G$ is finite, $G\cdot F$ is a finite $G$-invariant subgraph of $E$ containing $F$.  Thus $(S_{G,E})_0$ is the direct limit of the $(S_{G,F})_0$ with $F$ running over finite $G$-invariant subgraphs of $E$.  This reduces us to the case that $E$ is finite.    But if $G$ and $E$ are finite, then we can write $(S_{G,E})_0$ as the direct limit of the finite inverse subsemigroups $S_n$ with $n\geq 0$, where $S_n$ consists of $0$ and all elements $\alpha g\beta^*$ with $r(\alpha)=r(\beta)$, $|\alpha|=|\beta|\leq n$ and $g\in G$.  Thus $(S_{G,E})_0$ is locally finite when $G$ and $E$ are finite.  This completes the proof.
\end{proof}

Most famous examples of self-similar groups are finite state (like the Grigorchuk group) but few of the commonly studied ones are locally finite.  Of course, when $G$ is trivial we recover Hazrat's theorem on graded regularity of Leavitt path algebras in characteristic $0$~\cite{Hazrat14} for the special case of finite graphs with no sinks.

\section*{Appendix: A groupoid-free proof of Theorem~B}
In this appendix we give a groupoid-free proof of Theorem~B.

Let $S$ be an inverse semigroup, and let $R$ be a unital ring.  Then $S$ acts on the left of $E(S)$ via the rule $s\mathrel{\diamond} e = ses^*$.
The $\mathscr R$-class of $e\in E(S)$ is the set $R_e=\{s\in S\mid ss^*=e\}$.  Dually the $\mathscr L$-class of $e$ is $L_e=\{s\in S\mid s^*s=e\}$.   The maximal subgroup of $S$ at an idempotent $e$ is the group of units $G_e=R_e\cap L_e$ of the inverse monoid $eSe$ (with identity $e$).  Note that $G_e$ acts freely on the left of $R_e$ by multiplication since if $g\in G_e$ and $s\in R_e$, then we have $gss^*g^* = geg^*=e$, and $s=gs$ implies $e=ss^*=gss^*=ge=g$.

\begin{Lemma}\label{lemma:birkhoff}
Let $S$ be an semigroup and suppose that there is an integer $N>0$ with $|L_e|\leq N$ for all $e\in E(S)$.  Then $S$ is locally finite.
\end{Lemma}
\begin{proof}
The semigroup $S$ acts on the left of $L_e$  by partial injective mappings via the left Sch\"utzenberger representation:
\[s\cdot t = \begin{cases} st,  & \text{if}\ st\in L_e\\ \text{undefined}, & \text{else}\end{cases}\] for $s\in S$ and $t\in L_e$.   Moreover, the Sch\"utzenberger representations separate points~\cite{CP}.  That is, if $s_1,s_2\in S$ act the same in each Sch\"utzenberger representation, then $s_1=s_1\cdot s_1^*s_1 = s_2\cdot s_1^*s_1$, and so $s_1\leq s_2$.  Dually, we have $s_2\leq s_1$, and so they are equal.  It follows that the points of $S$ can be separated by homomorphisms into the symmetric inverse monoid $I_N$ of all partial bijections on an $N$-element set. Thus $S$ is locally finite by ~\ref{t:Birkhoff}.
\end{proof}

\begin{proof}[Proof of Theorem~B]
The sufficiency of these conditions was established in~\cite{Weissglass} and was proved in Proposition~\ref{p:inverse.semigroup.easy}.

Suppose next that $RS$ is regular.  The augmentation $\varepsilon\colon RS\to R$ given by on the $R$-basis $S$ by $s\mapsto 1$ is a surjective ring homomorphism, and hence $R$ is regular (cf.~\cite[Lemma~1.3]{Goodearlreg}).  Next observe that if $M=S^1$ is the inverse monoid resulting from adjoining an identity to $S$, then $RS$ is an ideal of $RM$ and we have that $R\cong RM/RS$.  Therefore $RM$ is regular by~\cite[Lemma~1.3]{Goodearlreg} as $R$ and $RS$ are regular.  Thus we may assume without loss of generality that $S$ is a monoid.

Suppose that $X$ is a finite subset of $S$ closed under inversion $x\mapsto x^*$.  We show that $T=\langle X\rangle$ is finite and the order of each maximal subgroup of $T$ is invertible in $R$.  Since every element of a maximal subgroup of $S$ belongs to some finitely generated inverse subsemigroup, it will follow that the order of each element of a maximal subgroup of $S$ is invertible in $R$.

There is a left module homomorphism $\pi\colon RS\to RE(S)$ given by $\pi(s) = ss^*$ for $s\in S$.  Indeed, $\pi(ts) = tss^*t^* = t\mathrel{\diamond} (ss^*) = t\mathrel{\diamond} \pi(s)$ for $s,t\in S$.  Evidently,  $\ker \pi$ is spanned over $R$ by the elements of the form $s-ss^*$ with $s\in S$.  Let $L\subseteq \ker \pi$ be the left ideal generated by the elements $t-tt^*$ with $t\in T$.  We claim that $L$ is generated by the finite set $Y=\{x-xx^*\mid x\in X\}$.  Indeed, we prove by induction on word length that if $t\in T$, then $t-tt^*\in RS\cdot Y$.  There is nothing to prove if $t\in X$.  Assume that $t=t_0x$ with $x\in X$ and $t_0$ having shorter word length.  Since $t_0^*$ has the same word length as $t_0$, because $X$ is closed under inversion, it follows by induction that $x-xx^*$ and $t_0^*-t_0^*t_0$ are in $RS\cdot Y$.  Therefore,
\[t-tt^* = t_0(x-xx^*) - t_0xx^*(t_0^*-t_0^*t_0)\in RS\cdot Y\] as required.  Since $RS$ is regular, it follows from~\cite[Theorem~1.1]{Goodearlreg} that there is an idempotent $e\in RS$ with $L=RSe$.  Let $f=1-e$ and note that $Lf = RSe(1-e)=0$.  Therefore,  $\pi(f)=\pi(1)-\pi(e) = 1$ (as $e\in L\subseteq \ker \pi$), and also $tf=tt^*f$ for all $t\in T$ (as $(t-tt^*)f=0$).  Let $f=\sum_{i=1}^Nr_is_i$ with $r_i\neq 0$ for all $i$.
We claim that $|L_x\cap T|\leq N$ for all $x\in E(T)$.  It will then follow that $T$ is finite by Lemma~\ref{lemma:birkhoff}.

Define\footnote{Note that $\wh f$ is essentially the image of $f$ under the isomorphism of $RS$ with $R\mathscr G_S$, cf.~\cite{mygroupoidalgebra}} $\wh f\colon S\to R$ by letting \[\wh f(s)=\sum_{s\leq s_i}r_i.\]  Notice that since $s\leq s_i$ if and only if $s=ss^*s_i$, if and only if $s=s_is^*s$, we deduce that $\wh f(s)$ is  coefficient of $s$ in $ss^*f$ and also the coefficient of $s$ in $fs^*s$.   If $x\in E(S)$, then there are at most $N$ elements in the support of $fx$, and so $\wh f$ has at most $N$ elements of $L_x$ in its support.

We claim that if $t\in T\cap L_{ss^*}$, then $\wh f(s) = \wh f(ts)$.  Indeed, $\wh f(s)$ is the coefficient of $s$ in $ss^*f$ and $\wh f(ts)$ is the coefficients of $ts$ in $tss^*t^*f=tss^*t^*tf=tss^*f$ because $t^*\in T$ implies $t^*f=t^*tf$.  But since $t^*t=ss^*$, left multiplication by $t$ is bijective mapping $ss^*S=t^*tS\to tt^*S=tS$ (with inverse left multiplication by $t^*$).  Therefore, the coefficient of $ts$  in $tss^*f$ is the same as the coefficient of $s$ in $ss^*f$, and so $\wh f(ts)=\wh f(s)$.

We may now complete the proof. Let $x\in E(T)$.  From
\begin{equation}\label{eq:augment2}
\sum_{i=1}^Nr_ixs_is_i^*x^*= \pi(xf) = x\mathrel{\diamond} \pi(f) = x\pi(f)x^* = x1x^*=x
\end{equation}
we deduce that there is $s$ in the support of $xf$ with $ss^*=x$.  Therefore, $\wh f(s)\neq 0$.  Now if $t\in L_x\cap T$, then $\wh f(ts)=\wh f(s)\neq 0$ by the above.  Since if $t\in L_x$, we have that $(ts)^*(ts) = s^*t^*ts=s^*ss^*s=s^*s$, all the elements $ts$ with $t\in L_x\cap T$ are in the support of $\wh f$, and so there are at most $N$ such elements.   On the other hand, if $t_1,t_2\in L_x$ and $t_1s=t_2s$, then $t_1=t_1x=t_1ss^*=t_2ss^*=t_2x=t_2$.  It follows that $L_x\cap T$ has at most $N$ elements.

Finally, we show that if $H_x=G_x\cap T\subseteq L_x\cap T$ is the maximal subgroup of $T$ at $x$, then $|H_x|$ is invertible in $R$.  Indeed, $H_x$ acts freely on the left of $R_x$.  Let $A$ be a transversal for $H_x\backslash R_x$.  Note that \eqref{eq:augment2} implies that
\[\sum_{s\in R_x} \wh f(s)=\sum_{xs_is_i^*x^*=x}r_i =1.\]  On the other hand, if $s\in R_x$, we already observed that $\wh f(ts)=\wh f(s)$ for all $t\in L_x\cap T\supseteq H_x$.  Since $H_x$ acts freely on the left of $R_x$, it follows that \[1=\sum_{s\in R_x}\wh f(s) = \sum_{a\in A}\sum_{h\in H_x}f(ha)=|H_x|\sum_{a\in A}f(a),\] and so $|H_x|$ is invertible in $R$.  This completes the proof.
\end{proof}

\def\malce{\mathbin{\hbox{$\bigcirc$\rlap{\kern-7.75pt\raise0,50pt\hbox{${\tt
  m}$}}}}}\def\cprime{$'$} \def\cprime{$'$} \def\cprime{$'$} \def\cprime{$'$}
  \def\cprime{$'$} \def\cprime{$'$} \def\cprime{$'$} \def\cprime{$'$}
  \def\cprime{$'$} \def\cprime{$'$}
  \def\malce{\mathbin{\hbox{$\bigcirc$\rlap{\kern-7.75pt\raise0,50pt\hbox{${\tt
  m}$}}}}}\def\cprime{$'$} \def\cprime{$'$} \def\cprime{$'$} \def\cprime{$'$}
  \def\cprime{$'$} \def\cprime{$'$} \def\cprime{$'$} \def\cprime{$'$}
  \def\cprime{$'$} \def\cprime{$'$}


\end{document}